\documentclass[11pt]{article}
\usepackage[a4paper, margin=1in]{geometry}
\usepackage{graphicx}
\usepackage{authblk}
\usepackage{hyperref}
\hypersetup{colorlinks=true,
            linkcolor=black,
            anchorcolor=black,
            citecolor=black}
\usepackage{caption}
\usepackage{amsmath,amsthm,amssymb}
\usepackage{mathrsfs}
\usepackage{amsfonts}
\usepackage{setspace}
\usepackage{multirow}
\usepackage[section]{placeins}
\usepackage{algorithm} 
\usepackage{algpseudocode}
\usepackage{bm}
\usepackage{cleveref}
\usepackage[authoryear]{natbib}

\setcitestyle{authoryear,open={(},close={)}}
\crefname{table}{Table}{Tables}
\crefname{equation}{Eq.}{Eqs.}

\usepackage{threeparttable}
\usepackage{booktabs} % partial horizontal lines with gaps in a table
\usepackage{setspace} % line spacing
\usepackage{tabularx}
\usepackage{float}
\usepackage{subcaption}
\newtheorem{lemma}{Lemma}
\newtheorem{proposition}{Proposition}
\newtheorem{definition}{Definition}

\newtheorem{remark}{Remark}
\usepackage{titlesec}
\usepackage{pifont}
\usepackage{arydshln}
\usepackage{rotating}
\usepackage[normalem]{ulem}

\title{Dominance-Based Feasibility Inference for Packing-Constrained Pickup and Delivery Problems}
\date{}
\author[1]{Siqiao Li}
\author[1]{Mahnam Saeednia}
\author[2]{Patrick Stokkink}
\affil[1]{\footnotesize Department of Transport and Planning, Faculty of Civil Engineering and Geosciences, Delft University of Technology, Delft, Netherlands}
\affil[2]{\footnotesize Department of Engineering Systems and Services, Faculty of Technology, Policy and Management, Delft University of Technology, Delft, Netherlands}

\newcommand{\route}{\kappa}
\newcommand{\placement}{\mathcal{Q}}

\begin{document}

\maketitle

\begin{abstract}
\normalsize
Routing and packing are intrinsically coupled in transport problems, requiring joint planning for cost-efficient and physically realizable solutions. We study a pickup and delivery problem with two-dimensional packing constraints (2P-PDP). Unlike vehicle routing variants where items are loaded before vehicles leave the depot and packing is validated only once, the 2P-PDP induces non-monotonic free-space evolution, substantially increasing feasibility-checking complexity. To address this bottleneck, we propose a generic dominance-based feasibility framework that is embeddable in a broad class of exact and heuristic routing algorithms. Under no-relocation constraints, inferring feasibility from a previously verified packing state requires preserving the pickup and delivery order of onboard items. To this end, we introduce an order-preserving mapping that jointly captures geometric containment and sequence compatibility, enabling dominance-based inference by embedding the new packing state into a verified reference plan. To further reduce dominance-screening overhead, we design three search rules to guide candidate exploration and tailored strategies to store, retrieve, and prioritize verified states. Computational experiments show that the proposed approach reduces feasibility-checking time by up to 42\% compared to a benchmark without dominance. The improvement stems from reducing exact packing-procedure calls, shifting verification effort away from the most computationally expensive stage.
\vspace{10pt}

\end{abstract} 

\section{Introduction}\label{sec:introduction}
In many vehicle routing problems, routing and packing decisions are tightly coupled: cost-efficient routes must also admit feasible arrangements of items within limited vehicle load space. The routing component, studied through variants of the Vehicle Routing Problem (VRP), determines the customer sequence, whereas the packing component, related to the Orthogonal Packing Problem (OPP), determines whether the associated items admit a feasible spatial arrangement within the vehicle.
Their coupling is not merely structural but causal: the visiting sequence dictates the order in which items are loaded and unloaded, which in turn governs how the available packing region evolves along the route. Ignoring this link produces routes that are cost-efficient on paper but physically unrealizable in execution, a failure mode that becomes increasingly costly in dynamic settings where routing plans must be recomputed repeatedly within an iterative search process. Yet jointly optimizing routing and packing is computationally challenging: both VRP and OPP are NP-hard in isolation, and their combination induces a combinatorial explosion in the joint search space. This motivates algorithmic frameworks capable of resolving packing feasibility efficiently without disrupting the broader routing search.

Despite this difficulty, a growing body of work has developed integrated routing-and-packing methods, predominantly for vehicle routing problems with two-dimensional loading constraints (2L-VRPs). In this setting, packages are loaded once at the depot before vehicle departure, and the main packing complication is to ensure that packages can be unloaded at each stop without displacing others, typically through a last-in--first-out discipline that keeps the packing subproblem relatively structured. See, e.g., \citet{iori2007exact,crainic2008extreme,zachariadis2012pallet,hokama2016branch,mahvash2017column,bortfeldt2020split,rajaei2022split,zhang2022branch_ts,ji2024branch,xu2025branch,ferreira2025exact,chi2025solving}. The pickup-and-delivery problem (PDP) setting is fundamentally more demanding. Since loading and unloading occur throughout route execution, the onboard load evolves non-monotonically. Moreover, pickup--delivery precedence requires each picked-up item to remain in its assigned position until delivery, while items being picked up or delivered must be accessible from the handling side. Packing feasibility therefore depends jointly on the loading sequence, the unloading sequence, and the route-dependent state of the available packing region: a three-way coupling absent in depot-loaded settings and not captured by standard 2L-VRP feasibility techniques. How to verify and exploit packing feasibility efficiently within this more demanding structure remains largely unresolved.

To close this gap, this paper studies the PDP with two-dimensional packing constraints (2P-PDP). The central computational difficulty is that packing feasibility must be verified repeatedly inside the routing algorithm, often at every candidate route evaluation, and these checks constitute the dominant computational bottleneck in practice. We address this by developing a dominance-based feasibility framework that exploits structural relationships among route-dependent packing states to avoid redundant feasibility checks. Rather than treating each feasibility check as an independent packing problem, the framework reuses information from previously verified states to infer feasibility for new states, substantially reducing the number of explicit packing checks required. The framework is designed as a self-contained module that can be embedded within both exact and metaheuristic routing algorithms without modifying the routing layer. While developed in the PDP context, the framework applies more broadly to routing variants in which packing feasibility must be verified repeatedly during route evaluation, such as dynamic variants. Our main contributions are as follows.

(1) We introduce a sequence-aware dominance relation over partial packing states that accounts for the dynamic, route-dependent evolution of the onboard load. This extends classical dominance arguments from static packing settings to the PDP context, where feasibility depends jointly on loading and unloading sequences, allowing redundant feasibility checks to be avoided through dominance.

(2) We develop an Order-Preserving Mapping (OPM) search tree that organizes feasible packing plans according to the ordering relationships among their associated loading and unloading sequences. This structure captures the combinatorial dependencies between spatial arrangements and temporal sequencing, providing a tractable search space for systematic dominance comparisons.

(3) We design structured storage, retrieval, and recency-aware sampling strategies for verified feasibility states. By prioritizing plausible and recently verified candidates, these strategies accelerate dominance-based inference, enabling the feasibility of unverified states to be determined through targeted comparisons and early termination.

The remainder of this paper is organized as follows. Section~\ref{sec:literature} reviews the related literature. Section~\ref{sec:problem} describes the 2P-PDP and its main operational constraints. Section~\ref{sec:alg} presents the dominance-based framework for packing feasibility verification, and Section~\ref{sec:strategies} introduces acceleration strategies for dominance checking. Section~\ref{sec:experiments} reports the computational experiments. Section~\ref{sec:conclusions} concludes the paper.

\section{Literature Review}\label{sec:literature}
Existing literature integrating routing and packing predominantly considers simplified VRP variants without pickup operations, where all packages are initially loaded at the depot and only the unloading sequence needs to be coordinated. Among studies that incorporate both pickup and delivery processes, the interaction between the two is often weakened by separating pickup and delivery requests: the depot acts as the delivery location for pickup requests and the pickup location for delivery requests \citep{zachariadis2016vehicle,reil2018heuristics,pinto2020variable}. Additional restrictions, such as a delivery-first--pickup-second service order, are also introduced to further reduce the coupling between the two operations \citep{bortfeldt2015hybrid}. Under these settings, packing feasibility along a route reduces to a small number of static checks.

Methodologically, these routing-and-packing problems are mainly addressed either by branch-and-cut-based approaches \citep{iori2007exact,hokama2016branch,zhang2022branch1,zhang2022branch_ts,xu2025branch,ferreira2021exact} or by heuristic approaches \citep{bortfeldt2012hybrid,wei2015variable,wei2018simulated,meliani2022tabu,krebs2023effective,ferreira2024variable,wang2025collaborative}. While effective for 2L-VRP variants involving only a limited number of packing checks, these approaches do not directly extend to PDP settings, where pickup and delivery operations are intertwined and packing feasibility must be verified dynamically along the route.

For PDP variants, onboard loads evolve dynamically along the route, making the feasibility of each subsequent operation dependent on all previous ones. To the best of our knowledge, only two studies address this issue, both under three-dimensional packing constraints. \citet{mannel2016hybrid} reduce dynamic packing verification by restricting each route to independent segments, so that all items picked up within a segment are also delivered within that segment and the vehicle is empty at the segment boundary. \citet{mannel2018solving} further relax this restriction and evaluate packing feasibility at identified checkpoints along a route. In both studies, routing decisions are obtained via a large neighborhood search and packing feasibility is assessed through a tree-search-based packing procedure.
A common feature of these approaches is that every encountered packing state is verified by solving a packing problem from scratch. While straightforward, this can lead to substantial computational overhead when many feasibility checks are required during route search. These approaches do not exploit relationships among route-dependent packing states; in particular, they do not use previously verified feasible states to support subsequent feasibility checks. In this paper, we address this gap through a dominance-based feasibility framework that formalizes when the feasibility of one packing state can be inferred from another, thereby reducing redundant packing verification during route search.

\section{Problem Description}\label{sec:problem}
The inputs consist of a node set $N$, a request set $R$, and a homogeneous vehicle fleet $V$. The node set includes a depot, from which vehicles depart and to which they must return, together with all pickup and delivery nodes. Each pair of nodes $(n, n') \in N \times N$ is associated with a travel distance $\tau_{n,n'}$. A request $r\in R$ is defined by a pickup node $p_r$, a delivery node $d_r$, and a rectangular item of length $l_r$ and width $w_r$. If a request involves multiple items, it is decomposed into item-level requests by duplicating the corresponding pickup and delivery nodes.

Vehicles have identical rectangular loading space of length $L$ and width $W$. They are rear-loaded, meaning that loading and unloading operations are performed exclusively from one end of the loading space.
Deploying a vehicle incurs a fixed cost $c^{\mathrm d}$, and traveling from nodes $n$ to $n'$ incurs an operational cost proportional to $\tau_{n,n'}$, with unit-distance cost $c^{\mathrm o}$. Each unserved request incurs a penalty cost $c^{\mathrm p}$. The 2P-PDP determines the fleet size and vehicle routes to minimize total deployment, operating, and penalty costs.

A feasible solution to the 2P-PDP must satisfy both routing and packing feasibility. Routing feasibility requires that (R1) each vehicle route starts and ends at the depot; (R2) each transportation request is assigned to exactly one vehicle, with its pickup and delivery nodes visited exactly once; and (R3) the total travel distance of each route does not exceed a prescribed limit. Additional routing-side restrictions, such as vehicle weight capacities, are omitted here but can be incorporated within the same routing framework in a straightforward manner. 

Packing feasibility requires that (P1) each item is placed entirely within the vehicle loading space and no two items overlap; (P2) items may be rotated horizontally by $90^\circ$ within the loading plane; and (P3) the no-relocation requirement is satisfied throughout route execution.
The no-relocation requirement reflects the operational setting in which items, once loaded, are not rearranged before their delivery. Such rearrangements would require extra handling during route execution, potentially delaying pickup and delivery services and increasing operational burden. This assumption is also common in packing-constrained vehicle routing studies \citep{Gendreau2006ts,iori2007exact,cote2014exact}.
The no-relocation requirement in (P3) has two rules: (P3a) the operation-accessibility rule, which applies to the item involved in a loading or unloading operation and requires that this operation be executable without relocating any other onboard item; and (P3b) the consistent-placement rule, which applies to items that remain onboard across consecutive operations and requires that their coordinates and orientations remain unchanged throughout their onboard intervals, regardless of changes in the set of co-loaded items. 

As illustrated in Figure~\ref{fig:packing_condition}, a Cartesian coordinate system is adopted to evaluate the packing conditions: the origin is located at the front-left corner of the loading space, the $x$-axis spans the width direction, and the $y$-axis spans the length direction. The length is bounded by $y=L$. Loading and unloading operations are performed from the rear end of the compartment; thus, larger $y$-coordinates indicate positions closer to the handling side.
Figure~\ref{fig:packing_condition} illustrates the no-relocation requirement under this coordinate system.
Constraint (P3a) is illustrated in Figures~\ref{fig:packing_condition}b and~\ref{fig:packing_condition}c, where items $i'$ and $i$ are accessible from the handling side for loading and unloading, respectively. Constraint (P3b) is illustrated by item $j$, whose coordinates and orientation remain unchanged across the feasible packing configurations.

\begin{figure}[htbp]
    \centering
    \includegraphics[width=\linewidth]{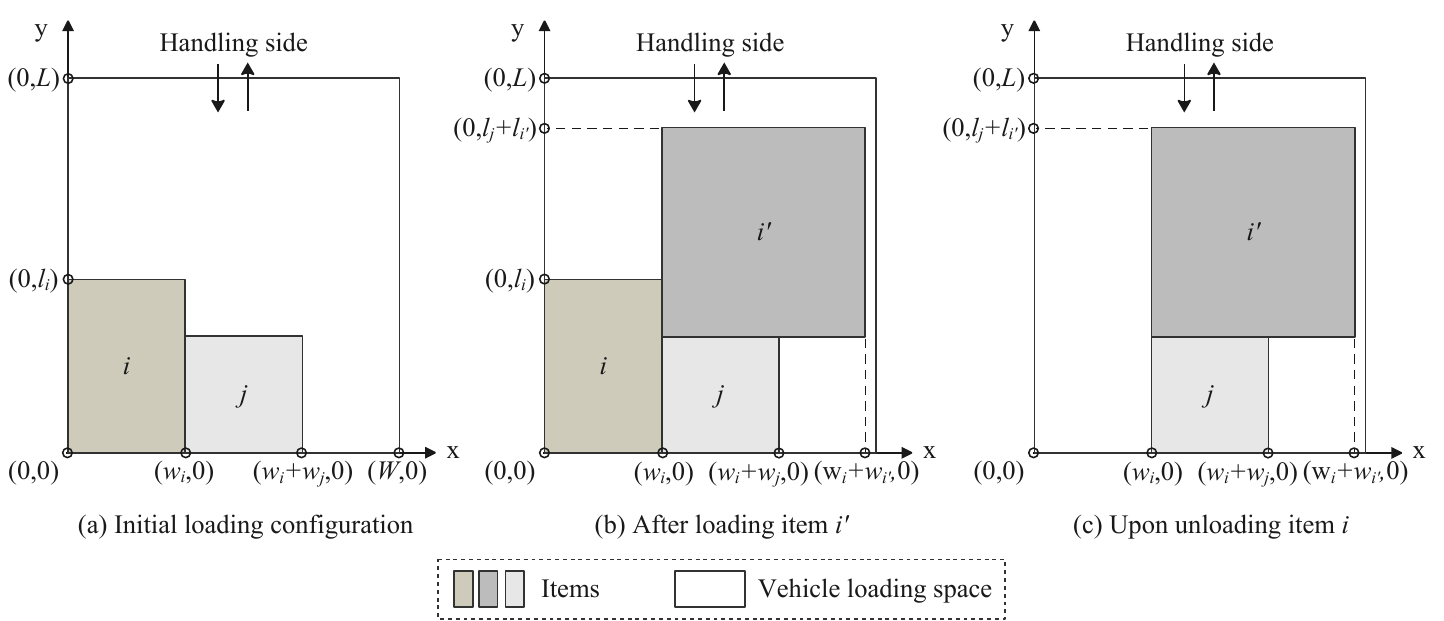}
    \caption{Illustration of the no-relocation requirement under the vehicle coordinate system.}
    \label{fig:packing_condition}
\end{figure}

Relative to a 2L-VRP, the additional computational difficulty of a 2P-PDP stems primarily from Constraint (P3). Since loading and unloading operations are interleaved along a route, with onboard loads changing after each event, Constraint (P3a) must hold at every operation to ensure accessibility of the relevant item placement. Meanwhile, Constraint (P3b) links these local feasibility requirements across the route by requiring items that remain onboard to preserve their coordinates and orientations over time. As a result, feasibility in a 2P-PDP depends not only on the packing configuration at individual operations, but also on its consistency over the entire sequence of route events.

The above conditions specify the routing and packing feasibility requirements considered in this study. Appendix~\ref{sec:app_exa} provides the constraint programming (CP) formulation corresponding to the packing constraints (P1)--(P3) for a given set of items and their pickup and delivery orders.

\section{Dominance-Based Framework for Packing Feasibility Verification}\label{sec:alg}
To address the computational bottleneck, this section develops a generic dominance-based framework that can be embedded within iterative solution methods in which routing decisions and packing-feasibility verification are handled as separate components.
Section~\ref{sec:alg_ver} presents a hierarchical feasibility-checking framework based on a limited set of necessary verifications. Section~\ref{sec:alg_dom} develops the proposed dominance-based screening method.

\subsection{Hierarchical Feasibility-Checking Framework}\label{sec:alg_ver}

Packing feasibility checks constitute the primary computational bottleneck. As packing feasibility evolves with the temporal changes of onboard loads in a PDP, an effective verification procedure should first identify the route positions at which packing states need to be checked. 
Section~\ref{sec:alg_ver_sopp} addresses this question by introducing the concept of Sequences of Open Pickup Points (SOPPs) together with a sequential verification procedure. 
Once the verification checkpoints are identified, the remaining challenge is to perform feasibility checks efficiently. To this end, we develop a hierarchical feasibility-checking scheme that limits the use of exact packing routines. Section~\ref{sec:alg_ver_hier} presents this hierarchical scheme.

\subsubsection{Identification of Verification Checkpoints} \label{sec:alg_ver_sopp}

For a given route, the packing state tends to become more constrained as items are picked up and less constrained as items are delivered. Thus, not every route position introduces a distinct packing state: some positions need not be checked separately because their feasibility is implied by more restrictive states encountered along the same sequence. Verifying packing feasibility at every route position is therefore unnecessary and computationally inefficient. This section identifies a limited set of checkpoints that capture these most restrictive packing states.

To this end, we adopt the SOPP concept introduced by \citet{mannel2018solving}, which identifies pickup sequences whose items remain simultaneously onboard and therefore interact in the packing configuration.

\begin{definition}[Sequence of open pickup points]
\label{def:sopp}
A Sequence of Open Pickup Points (SOPP) in a route is a sequence of pickup nodes satisfying: (i) the last pickup node in the sequence is immediately followed by a delivery node in the route; and (ii) the sequence contains exactly all open pickup nodes, up to and including its last pickup node, whose corresponding delivery nodes appear after this last pickup node in the route.
\end{definition}

By definition, pickup nodes whose items have already been delivered are excluded from an SOPP. Hence, all items associated with an SOPP remain undelivered after its last pickup and are therefore simultaneously onboard.
The most restrictive packing state associated with the SOPP therefore occurs at this position, which motivates Definition~\ref{def:non-dominated_arc}.

\begin{definition}[Non-dominated packing arc]
\label{def:non-dominated_arc}
Given an SOPP, its associated Non-Dominated Packing Arc (NDPA) is the arc immediately following the last pickup node in the sequence.
\end{definition}

The onboard load at an NDPA consists of all items picked up but not yet delivered at that point. Consequently, this arc represents a non-dominated packing state, and feasibility verification at the NDPA is sufficient to characterize the SOPP: earlier positions represented by subsequences of the SOPP involve only subsets of the open items present at the NDPA.
An illustration of a route and its three SOPPs is provided in \autoref{fig:sopp}. For the NDPA of each SOPP, the corresponding local packing plan must satisfy Constraints~(P1)--(P3a). In addition, request $r_1$ remains onboard throughout the three SOPPs, so Constraint~(P3b) requires its position and orientation to be identical across the corresponding local packing plans.

\begin{figure}[htbp]
    \centering
    \includegraphics[width=0.9\linewidth]{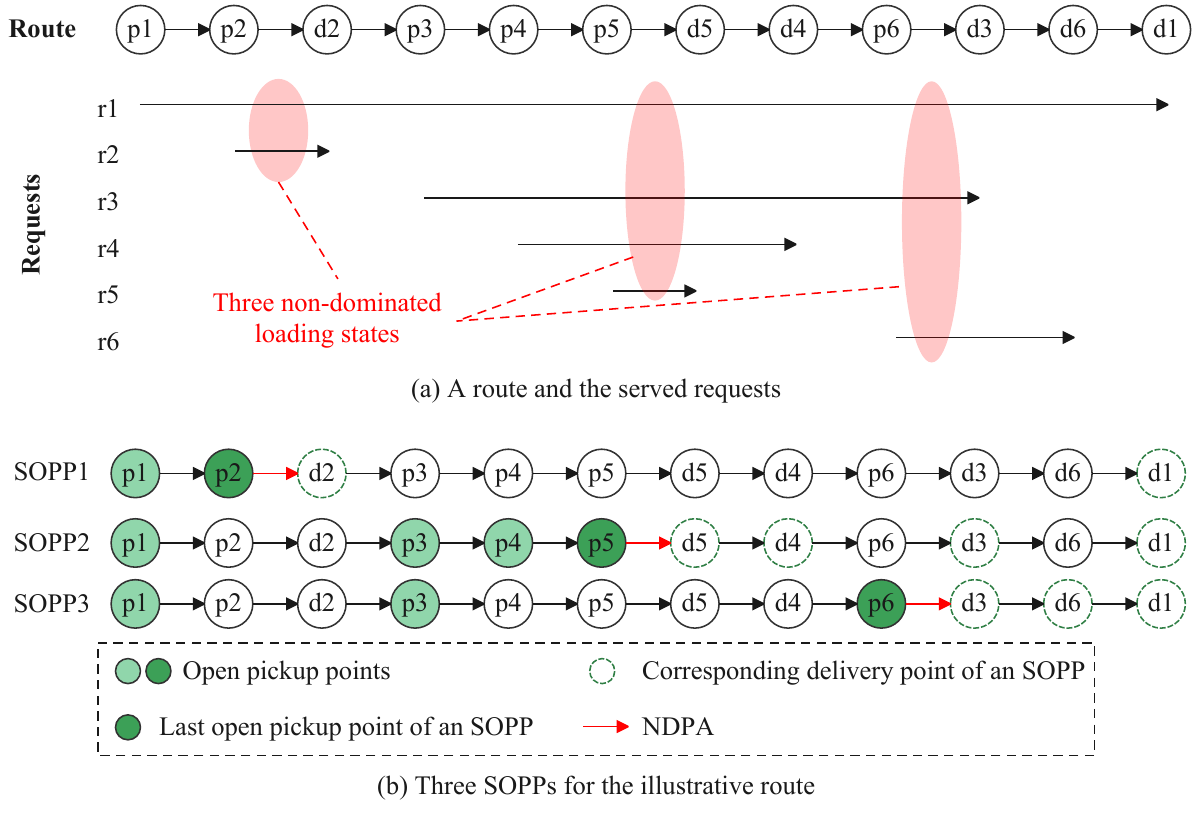}
    \caption{An illustrative route and its SOPPs}
    \label{fig:sopp}
\end{figure}

The above definitions localize packing feasibility to a finite set of critical packing states. Proposition~\ref{prop:route_feasibility} connects these local checks to route-level packing feasibility.

\begin{proposition}[Route packing feasibility via sequential SOPP verification]
\label{prop:route_feasibility}
Let $\route=(n_0,n_1,\ldots,n_m)$ be a route and let $A(\route)=(a_0,a_1,\ldots,a_{m-1})$, where $a_\ell=(n_\ell,n_{\ell+1})$ for $\ell=0,\ldots,m-1$, denote its induced arc sequence. Let $\Psi_\route=(\psi_1,\psi_2,\ldots,\psi_k)$ denote the ordered collection of SOPPs associated with route $\route$, indexed such that $\psi_i$ precedes $\psi_{i+1}$ if and only if the NDPA of $\psi_i$ is $a_\ell$ and the NDPA of $\psi_{i+1}$ is $a_{\ell'}$ with $\ell<\ell'$.
Route $\route$ is packing-feasible if and only if: (i) every $\psi_i\in\Psi_\route$ is feasible at its NDPA; and (ii) for every pair of SOPPs $\psi_i$ and $\psi_{j}$, each shared item has the same position and orientation in their packing plans.
\end{proposition}
\begin{proof}
Let $I(\psi_i)$ denote the onboard item set of SOPP $\psi_i$ at its NDPA, and let $\pi_i(q)$ denote the position and orientation assigned to item $q\in I(\psi_i)$ in the packing plan of $\psi_i$.\\
($\Rightarrow$) Suppose route $\route$ is packing-feasible, i.e., Constraints (P1)--(P3) hold. Then every packing state along $\route$ satisfies the packing constraints. In particular, the packing state at the NDPA of each $\psi_i\in\Psi_\route$ is feasible, so condition~(i) holds. By the consistent-placement rule~(P3b), for any $\psi_i,\psi_j\in\Psi_\route$ and any $q\in I(\psi_i)\cap I(\psi_j)$, the position and orientation of item $q$ cannot change. Hence condition~(ii) holds.\\
($\Leftarrow$) Suppose conditions~(i) and~(ii) hold. By condition~(i), every SOPP admits a feasible packing plan at its NDPA, which verifies the most restrictive packing state represented by that SOPP. Earlier route positions represented by subsequences of the SOPP involve only subsets of the open items present at the NDPA. Therefore, restricting the NDPA packing plan to the corresponding item subset preserves containment and non-overlap, so Constraints~(P1) and~(P2) hold along the route. Moreover, because the SOPP feasibility check incorporates operation accessibility for the route operations represented by that SOPP, Constraint~(P3a) holds along the route. 
By condition~(ii), the packing plans of all SOPPs agree on every shared item; hence, these local plans can be combined along the route without relocating any item that remains onboard across multiple SOPPs, thereby satisfying Constraint~(P3b). Therefore, all packing constraints are satisfied throughout $\route$, and route $\route$ is packing-feasible.
\end{proof}

\begin{remark}
Proposition~\ref{prop:route_feasibility} justifies a sequential verification of packing feasibility: SOPPs are checked in the order $\psi_1, \psi_2, \ldots, \psi_k$. As soon as any $\psi_i$ is packing-infeasible or a placement inconsistency is detected, the route is declared infeasible and verification terminates. 
\end{remark}

\subsubsection{Hierarchical Procedure for Packing Feasibility Verification}\label{sec:alg_ver_hier}
Algorithm~\ref{alg:feasibility_check} presents the hierarchical feasibility-checking framework for an unverified route under the packing conditions defined in Section~\ref{sec:problem}. The procedure is organized in increasing order of expected computational burden. It first applies a conservative heuristic packing check, which is inexpensive and can quickly verify packing states. If the heuristic check is inconclusive, dominance-based screening is used to infer feasibility from previously verified packing states and thereby avoid exact verification when possible. An exact method is invoked only when both the heuristic check and dominance screening fail to establish feasibility. This ordering reflects the role of the hierarchy: inexpensive tests are used to resolve easy cases, while the computationally demanding exact packing procedure is reserved for cases that cannot be verified otherwise. The relative order of the heuristic and dominance checks may be adjusted depending on their expected computational burden; in particular, dominance screening can be placed first when heuristic verification is comparatively expensive. The framework can accommodate any exact and conservative heuristic packing procedures; in this study, these are instantiated by a CP approach and a maximum-open-space (MOS)-based heuristic, respectively, as detailed in Appendices~\ref{sec:app_exa} and~\ref{sec:app_heu}.

\begin{algorithm}[htbp]
\caption{Hierarchical packing-feasibility check for a route}
\label{alg:feasibility_check}
\begin{algorithmic}[1]
\State \textbf{Input:} unverified route $\route$; verified SOPP repository $\bar \Psi$
\State $\Psi_\route \gets$ ordered collection of SOPPs associated with route $\route$
\For{each $\psi \in \Psi_\route$}
    \If{the total area of items in $\psi$ exceeds the vehicle loading space}
        \State \Return False
    \EndIf
    \If{\textsc{HeuristicPack}$(\psi)$ returns feasible} \Comment{Appendix~\ref{sec:app_heu}}
        \State \textbf{continue}
    \ElsIf{$\exists\,\bar{\psi}\in\bar{\Psi}$ satisfying the sequence-aware dominance condition with respect to $\psi$}
    \Comment{Section~\ref{sec:alg_dom}}
        \State \textbf{continue}
    \ElsIf{\textsc{ExactPack}$(\psi)$ returns feasible} \Comment{Appendix~\ref{sec:app_exa}}
        \State $\bar \Psi \gets \bar \Psi \cup \{\psi\}$
    \Else
        \State \Return False
    \EndIf
\EndFor
\State \Return True
\end{algorithmic}
\end{algorithm}

\subsection{Dominance-Based Screening Method}\label{sec:alg_dom}
\subsubsection{Dominance Definition via Order-Preserving Mappings}
The difficulty of dominance checking lies in jointly accounting for geometric packing feasibility and sequence-dependent accessibility. Under rear-end accessibility, the no-relocation constraint~(P3a) is sequence-dependent in two loading directions. On the pickup side, items picked up later must remain loadable when inserted into the compartment, which leads to the reverse pickup order. On the delivery side, items delivered earlier must remain unloadable when removed, which leads to the delivery order. We refer to these two requirements as the pickup-side and delivery-side no-relocation conditions, denoted by NR-p and NR-d, respectively.

To exploit dominance, we interpret a verified feasible packing plan as a reference structure for checking a new packing state. At the geometric level, dominance can be viewed as a static bin--piece embedding problem: the placed items in the verified plan define reference bins, while the items in the new packing state are treated as pieces to be embedded into these bins. 
Figure~\ref{fig:bin_piece_example} illustrates the bin--piece interpretation using two separate mappings, one under the delivery-side order required by NR-d and one under the pickup-side order required by NR-p. Without sequence-dependent accessibility, geometric containment alone would be sufficient to infer feasibility of the new packing state. In the 2P-PDP, however, dominance verification must also preserve the relevant loading and unloading orders. This motivates the order-preserving mapping introduced in Definition~\ref{def:opm}, which integrates geometric containment with sequence compatibility.

\begin{figure}[H]
    \centering
    \includegraphics[width=0.9\linewidth]{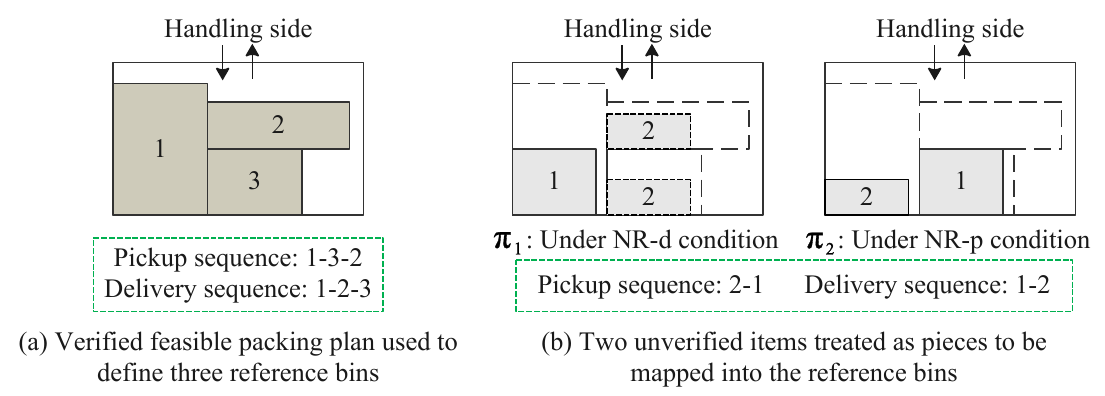}
    \caption{Bin--piece interpretation for dominance checking under separate NR-p and NR-d conditions. Under NR-d, piece $2$ can be mapped to two reference bins, yielding two candidate placements.}
    \label{fig:bin_piece_example}
\end{figure}

\begin{definition}[Order-preserving mapping]
\label{def:opm}
Let $\mathcal{I}$ and $\mathcal{J}$ denote two item sequences induced by two SOPPs, where both sequences are indexed according to the same rule, either in the reverse order of pickup or in the order of delivery. We interpret $\mathcal{I}$ as a sequence of bins and $\mathcal{J}$ as a sequence of pieces. A function $f:\mathcal{J}\to\mathcal{I}$ is called an Order-Preserving Mapping (OPM) if, for all consecutive pieces $j,j+1\in \mathcal{J}$, it satisfies: $f(j+1)\ge f(j)$. That is, $f$ preserves the relative order of pieces with respect to the bins.
Given an OPM $f$, a packing plan specifies the placement of each piece within the spatial region associated with its assigned bin. 
\end{definition}

An OPM partitions the sequence of pieces into ordered groups, each assigned to one reference bin. The monotonicity condition prevents a later piece in the sequence from being mapped to a bin that precedes the bin assigned to an earlier piece. Thus, the mapping captures the sequence consistency required by either NR-p or NR-d, while the within-bin placement checks geometric containment and local accessibility. Because these bin regions are inherited from a feasible reference packing and are non-overlapping, feasibility within each assigned bin is sufficient to rule out additional spatial conflicts across bins. Lemma~\ref{lemma:dom_d} establishes the resulting partial dominance relation when one of the two no-relocation conditions is considered.

\begin{lemma}[Sequence-aware partial dominance]
\label{lemma:dom_d}
Consider two SOPPs inducing two item sequences, $\mathcal{I}$ and $\mathcal{J}$, indexed according to the same rule. Suppose that a feasible packing plan exists for $\mathcal{I}$ on the associated NDPA, and that there exists an OPM $f:\mathcal{J}\to\mathcal{I}$ such that each piece in $\mathcal{J}$ can be feasibly placed within the region of its mapped reference bin in $\mathcal{I}$ while satisfying Constraints~(R1)--(R3), (P1), (P2), and (P3a) within each bin. 
Then:
(i) if both sequences are indexed in the reverse order of pickup, $\mathcal{J}$ admits a feasible packing plan satisfying NR-p; and 
(ii) if both sequences are indexed in the order of delivery, $\mathcal{J}$ admits a feasible packing plan satisfying NR-d.
\end{lemma}
\begin{proof}
Since a feasible packing plan exists for $\mathcal{I}$, the corresponding items satisfy the geometric feasibility conditions in Constraints~(R1)--(R3), (P1), and (P2), as well as the relevant accessibility requirement in Constraint~(P3a).\\
For part~(i), let both sequences be indexed in the reverse order of pickup. Because $f$ is order-preserving, the pickup precedence relation among pieces in $\mathcal{J}$ is consistent with that among their mapped bins in $\mathcal{I}$. Moreover, each piece is feasibly placed within its mapped bin while satisfying Constraints~(R1)--(R3), (P1), (P2), and (P3a) locally. Therefore, the accessibility relation required by NR-p is inherited from $\mathcal{I}$ to $\mathcal{J}$, and $\mathcal{J}$ admits a feasible packing plan satisfying NR-p.
Part~(ii) follows analogously when both sequences are indexed in the order of delivery. Following the same logic, $\mathcal{J}$ admits a feasible packing plan satisfying NR-d.
\end{proof}

Lemma~\ref{lemma:dom_d} provides a criterion for verifying whether the NR-p or NR-d condition holds for an SOPP through partial dominance. Proposition~\ref{prop:domi} strengthens the partial dominance result to obtain a dominance criterion that simultaneously enforces both no-relocation conditions.

\begin{proposition}[Sequence-aware dominance]
\label{prop:domi}
Consider two SOPPs with item sets $\mathcal{I}$ and $\mathcal{J}$, where $\mathcal{I}$ is associated with a verified feasible packing plan. Let $\mathcal{I}^\mathrm{p}$ and $\mathcal{J}^\mathrm{p}$ denote the corresponding item sequences indexed in the reverse order of pickup, and let $\mathcal{I}^\mathrm{d}$ and $\mathcal{J}^\mathrm{d}$ denote those indexed in the order of delivery.
If there exist two OPMs, $f^\mathrm{p}:\mathcal{J}^\mathrm{p}\to\mathcal{I}^\mathrm{p}$ and $f^\mathrm{d}:\mathcal{J}^\mathrm{d}\to\mathcal{I}^\mathrm{d}$, such that:\\
(i) $f^\mathrm{p}$ and $f^\mathrm{d}$ satisfy the sequence-aware partial dominance condition in Lemma~\ref{lemma:dom_d} for the NR-p and NR-d conditions, respectively; and\\
(ii) the two mappings induce a common packing plan for $\mathcal{J}$, i.e., each item in $\mathcal{J}$ has the same position and orientation under the two induced packing realizations,\\
then $\mathcal{J}$ is dominated by $\mathcal{I}$ and admits a feasible packing plan satisfying both the NR-p and NR-d simultaneously.
\end{proposition}
\begin{proof}
By condition~(i) and Lemma~\ref{lemma:dom_d}(i), the mapping $f^\mathrm{p}$ induces a feasible packing realization $\pi^\mathrm{p}$ for $\mathcal{J}$ satisfying NR-p. By Lemma~\ref{lemma:dom_d}(ii), the mapping $f^\mathrm{d}$ induces a feasible packing realization $\pi^\mathrm{d}$ for $\mathcal{J}$ satisfying NR-d. In general, $\pi^\mathrm{p}$ and $\pi^\mathrm{d}$ need not coincide.\\
By condition~(ii), however, $f^\mathrm{p}$ and $f^\mathrm{d}$ assign to each item $j \in \mathcal{J}$ the same position and orientation, so $\pi^\mathrm{p} = \pi^\mathrm{d}$; call this common realization $\pi^*$. Since $\pi^*$ equals $\pi^\mathrm{p}$, it satisfies NR-p; since it equals $\pi^\mathrm{d}$, it satisfies NR-d. Furthermore, $\pi^*$ satisfies Constraints~(R1)--(R3) and (P1)--(P2) because both $\pi^\mathrm{p}$ and $\pi^\mathrm{d}$ do by Lemma~\ref{lemma:dom_d}. Hence $\pi^*$ is a single feasible packing plan for $\mathcal{J}$ satisfying all required constraints, and $\mathcal{J}$ is dominated by $\mathcal{I}$.
\end{proof}

\autoref{fig:dom_pd} illustrates the distinction between partial and full dominance. SOPP1 provides a verified feasible packing plan with three reference bins. For SOPP2, realization $\pi_1$ satisfies only NR-d; items 4 and 5, which are picked up later than item 3, are not accessible for loading. Realization $\pi_2$ represents the common plan $\pi^*$ that simultaneously satisfies both NR-p and NR-d, as required by condition~(ii) of Proposition~\ref{prop:domi}.
\begin{figure}[H]
    \centering
    \includegraphics[width=0.75\linewidth]{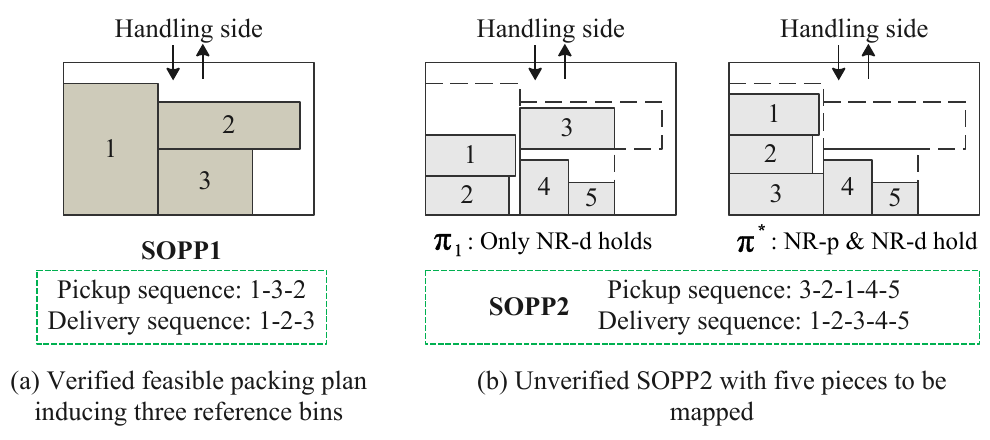}
    \caption{Illustration of partial dominance under the NR-d condition and its extension to full dominance.}
    \label{fig:dom_pd}
\end{figure}

\subsubsection{OPM Search Tree and Pruning Rules}\label{sec:opm_tree}
Dominance screening requires identifying an OPM that is feasible under both NR-p and NR-d. However, the number of possible mappings grows combinatorially with the numbers of bins and pieces. In particular, given $|\mathcal{I}|$ bins and $|\mathcal{J}|$ pieces, the number of possible mapping functions is $\binom{|\mathcal{I}|+|\mathcal{J}|-1}{|\mathcal{J}|}$.
Exhaustive enumeration therefore becomes prohibitive even for moderate instance sizes. To address this, we adopt an incremental tree search in which OPMs are constructed piece by piece, and infeasible partial mappings are pruned as early as possible.

The tree search is built on the notion of a partial OPM. Definition~\ref{def:partial_opm} formalizes partial mappings that preserve the monotonicity requirement. Definition~\ref{def:opm_tree} then specifies the resulting search-tree structure induced by incremental assignments.

\begin{definition}[Partial OPM]
\label{def:partial_opm}
Given ordered sequences of pieces $\mathcal{J}$ and bins $\mathcal{I}$, a partial OPM is an OPM defined on a subset of pieces $\mathcal{J}'\subseteq\mathcal{J}$. That is, it is a function $f:\mathcal{J}'\to\mathcal{I}$ such that, for any $j_1,j_2\in\mathcal{J}'$ with $j_1 < j_2$, we have $f(j_1)\le f(j_2)$. It represents a partially constructed mapping in which only the pieces in $\mathcal{J}'$ have been assigned to bins.
\end{definition}

\begin{definition}[OPM search tree]
\label{def:opm_tree}
Given ordered sequences $\mathcal{J}=(j_1,\dots,j_{|\mathcal{J}|})$ and $\mathcal{I}=(i_1,\dots,i_{|\mathcal{I}|})$, the OPM search tree is a rooted tree of depth $|\mathcal{J}|$. Each node at depth $\rho$ represents a partial OPM $f_\rho:\{j_1,\dots,j_\rho\}\to\mathcal{I}$, where the root corresponds to $\rho=0$ and represents the empty mapping. 
At depth $\rho=1$, each node is obtained by assigning $j_1$ to a bin in $\mathcal{I}$. For $\rho \geq 1$, each node is obtained by extending a node at depth $\rho-1$ through assigning $j_\rho$ to a bin index not smaller than that of $j_{\rho-1}$, i.e., $f_\rho(j_\rho)\geq f_{\rho-1}(j_{\rho-1})$.
\end{definition}

To support pruning, it is useful to distinguish bins whose assigned piece sets are fixed from those that may still receive additional pieces. For a node at depth $\rho$, let $\mathcal{I}_\rho=\{\,f_\rho(j_k):1\le k\le \rho\,\}$ denote the set of bins that have received at least one piece. Definition~\ref{def:open_bins} classifies these bins into open and closed bins.

\begin{definition}[Open and closed bins]
\label{def:open_bins}
At depth $\rho\geq 1$, the bin $f_\rho(j_\rho)$ is called the \emph{open bin}. It remains available for future mappings, since subsequent pieces at deeper levels may also be assigned to it. Any bin in $\mathcal{I}_\rho\setminus\{f_\rho(j_\rho)\}$ is called a \emph{closed bin}. Its packing state is fixed and it will not receive any additional pieces in subsequent extensions.
\end{definition}

Definition~\ref{def:open_bins} implies that extending a node can only add pieces to the current open bin or to new empty bins with larger indices; the contents of closed bins are inherited unchanged. Since the piece set of a closed bin is fully determined at the moment it becomes closed, its packing feasibility can be verified immediately at that tree node. If a closed bin is found infeasible, all descendants of that node can be pruned, because subsequent extensions cannot modify its assigned piece set and therefore cannot restore its feasibility. Proposition~\ref{prop:prune} formalizes this pruning condition.

\begin{proposition}[Pruning by closed-bin infeasibility]
\label{prop:prune}
Consider a tree node at depth $\rho$ with partial mapping $f_\rho$. If any closed bin of $f_\rho$ is infeasible under the packing constraints, then no complete OPM extending $f_\rho$ is feasible, and the entire subtree rooted at this node can be pruned.
\end{proposition}
\begin{proof}
Let $i^* \in \mathcal{I}_\rho \setminus \{f_\rho(j_\rho)\}$ be a closed bin that is infeasible under the packing constraints. By Definition~\ref{def:open_bins}, $i^*$ receives no additional pieces in any extension $f_{\rho'}$ of $f_\rho$ at depth $\rho' > \rho$, so the piece set of $i^*$ under $f_{\rho'}$ is identical to that under $f_\rho$. Hence $i^*$ remains infeasible for every such extension. Since a feasible complete OPM requires every bin to admit a feasible packing, no completion of $f_\rho$ is feasible, and the subtree can be pruned.
\end{proof}

\subsubsection{Search Initialization and Tree-Exploration Strategy}\label{sec:tree_cut}

Building on the OPM search tree in Section~\ref{sec:opm_tree}, the remaining task is to specialize the search to the dominance condition in Proposition~\ref{prop:domi}. Complete dominance requires compatibility between the pickup-sorted and delivery-sorted mappings together with a consistent in-bin packing plan. In addition, when the SOPP under verification shares items with preceding SOPPs on the same route, Proposition~\ref{prop:route_feasibility} requires that the placements of these shared items remain unchanged, which is exactly the role of Constraint~(P3b). Therefore, the OPM search must account for both cross-order compatibility and inherited placements. Since the goal is to verify the existence of at least one feasible realization rather than to optimize over all OPMs, computational efficiency relies on restricting the search space as early as possible and exploring promising branches first. To this end, three search rules are introduced: a monotonicity pruning rule, a consistency pruning rule, and a sibling-node dominance rule.

\paragraph{(1) Monotonicity pruning rule.}
The OPM search tree is constructed under one specified order, either the pickup-sorted order or the delivery-sorted order. The monotonicity condition associated with that order is enforced directly by the tree structure in Definition~\ref{def:opm_tree}. However, complete dominance in Proposition~\ref{prop:domi} requires compatibility with the other order as well. Therefore, whenever a node is extended under the current order, its partial mapping must also be checked against the monotonicity requirement induced by the other order. If this additional requirement is violated, then the corresponding branch cannot lead to a mapping that is monotone under both orders and can be discarded immediately.

\begin{proposition}[Monotonicity pruning]
\label{prop:mono_cut}
Consider an OPM search tree constructed under one specified order, either pickup-sorted or delivery-sorted. If a partial OPM at any node violates the monotonicity requirement induced by the other order, then no extension of this partial OPM can satisfy both orders, and the corresponding branch can be pruned.
\end{proposition}
\begin{proof}
Let $f_\rho$ be a partial OPM at depth $\rho$ that violates the monotonicity requirement induced by the other order. Then there exist two pieces $j_s, j_t \in \mathcal{J}'$ with $s < t$ under the other order such that $f_\rho(j_s) > f_\rho(j_t)$. Any extension $f_{\rho'}$ of $f_\rho$ at depth $\rho' > \rho$ satisfies $f_{\rho'}(j_s) = f_\rho(j_s)$ and $f_{\rho'}(j_t) = f_\rho(j_t)$ by definition of extension, so $f_{\rho'}(j_s) > f_{\rho'}(j_t)$ holds for every such extension, regardless of how remaining pieces are assigned. Hence no completion of $f_\rho$ is monotone under the other order, and the branch can be safely pruned.
\end{proof}

The monotonicity pruning rule depends only on the two item orders and can therefore be evaluated before invoking geometric feasibility checks. In this sense, it serves as a structural filter that restricts the search to mappings that remain compatible with both orderings.

\paragraph{(2) Consistency pruning rule.}
The second restriction comes from route-level feasibility. By Proposition~\ref{prop:route_feasibility}, a route is packing-feasible only if all of its SOPPs are feasible and the placements of items shared across consecutive SOPPs remain unchanged. Thus, when the SOPP under verification has predecessors on the same route, the OPM search must exclude any branch for which no packing realization can preserve the inherited placements of these shared items. This is the role of the consistency pruning rule.

Let $\mathcal{J}^{\mathrm{fix}}\subseteq\mathcal{J}$ denote the set of shared pieces whose placements are inherited from preceding SOPPs and therefore fixed. For each $j\in\mathcal{J}^{\mathrm{fix}}$, let $\phi(j)\in\mathcal{I}$ denote the bin in the verified SOPP that contains the inherited placement of piece $j$. Then any feasible realization of complete dominance must map $j$ to $\phi(j)$, where the coordinates are fixed.

\begin{proposition}[Consistency pruning]
\label{prop:cons_cut}
Consider a node in the OPM search tree with partial mapping $f_\rho$ and let $\mathcal{J}^{\mathrm{assigned}}_\rho \subseteq \mathcal{J}$ denote the set of pieces assigned by $f_\rho$. If there exists a piece $j \in \mathcal{J}^{\mathrm{fix}} \cap \mathcal{J}^{\mathrm{assigned}}_\rho$ such that $f_\rho(j) \neq \phi(j)$, then the node is infeasible and can be pruned.
\end{proposition}
\begin{proof}
Let $j \in \mathcal{J}^{\mathrm{fix}} \cap \mathcal{J}^{\mathrm{assigned}}_\rho$ with $f_\rho(j) \neq \phi(j)$. By Proposition~\ref{prop:route_feasibility}, any feasible complete OPM must assign $j$ to $\phi(j)$ with its inherited coordinates. Since any extension $f_{\rho'}$ of $f_\rho$ satisfies $f_{\rho'}(j) = f_\rho(j) \neq \phi(j)$ by definition of extension, no completion of $f_\rho$ can satisfy this requirement. Hence the node is infeasible and can be pruned.
\end{proof}

The consistency pruning rule therefore introduces Constraint~(P3b) from Proposition~\ref{prop:route_feasibility} directly into the OPM search. Together, Propositions~\ref{prop:mono_cut} and~\ref{prop:cons_cut} restrict the candidate mappings to those that are compatible with both orderings and with inherited placements.

\paragraph{(3) Sibling-node dominance rule.}
After applying the two pruning rules above, the remaining search still depends on the order in which admissible branches are explored. Since the procedure terminates as soon as a feasible OPM is found, search order affects efficiency but not the feasibility outcome. Recall that the pieces and bins are indexed according to the ordered sequences used to construct the OPM search tree, either pickup-sorted or delivery-sorted; hence, $i<i'$ means that bin $i$ precedes bin $i'$ in the ordered bin sequence. The following rule justifies exploring sibling nodes in increasing order of their bin positions in the ordered sequence $\mathcal{I}$.

\begin{proposition}[Sibling-node dominance]
\label{prop:sibling}
Consider two sibling nodes at depth $\rho$ whose partial mappings $f_\rho$ and $f_\rho'$ differ only in the assignment of piece $j_\rho$, with $f_\rho(j_\rho) = i$ and $f_\rho'(j_\rho) = i'$ for $i < i'$. Suppose that $f_\rho(j_k) = f_\rho'(j_k)$ for all previously assigned pieces $j_k$, $k < \rho$. Consider the two child nodes $f_{\rho+1}$ and $f_{\rho+1}'$ obtained by extending $f_\rho$ and $f_\rho'$, respectively, by assigning piece $j_{\rho+1}$ to bin $i'$, with all other assignments unchanged. If $f_{\rho+1}$ is infeasible at bin $i'$, then $f_{\rho+1}'$ is also infeasible.
\end{proposition}
\begin{proof}
The piece set of bin $i'$ under $f_{\rho+1}$ consists of piece $j_{\rho+1}$ together with all pieces assigned to $i'$ by the common assignments of previously assigned pieces $\{f_\rho(j_k) : k < \rho\}$, but excludes piece $j_\rho$ since $f_\rho(j_\rho) = i \neq i'$. The piece set of bin $i'$ under $f_{\rho+1}'$ contains the same pieces plus $j_\rho$, since $f_\rho'(j_\rho) = i'$. 
Hence the piece set of bin $i'$ under $f_{\rho+1}'$ is a strict superset of the piece set under $f_{\rho+1}$, namely, it contains all pieces in the latter set plus $j_\rho$.
Since feasibility of a bin requires a valid packing of its piece set, and any feasible packing of a strict superset would also yield a feasible packing of the subset, infeasibility of bin $i'$ under $f_{\rho+1}$ implies infeasibility of bin $i'$ under $f_{\rho+1}'$. Hence $f_{\rho+1}'$ is infeasible and can be pruned.
\end{proof}
\begin{remark}
\label{remark:sibling}
Proposition~\ref{prop:sibling} motivates exploring sibling nodes in increasing order of their bin positions in $\mathcal{I}$, i.e., sibling nodes mapping $j_\rho$ to earlier bins in $\mathcal{I}$ are explored first. Once a node $f_{\rho+1}$ is found infeasible at bin $i'$, the corresponding node $f_{\rho+1}'$ generated from the sibling parent that maps $j_\rho$ to a later bin $i' $ can be pruned without evaluation.
\end{remark}

Figure~\ref{fig:opm_tree} illustrates the search-tree structure and the pruning effects induced by Propositions~\ref{prop:mono_cut}--\ref{prop:cons_cut}. Tree node~4 is pruned since piece $j_2$ is mapped to a bin with a smaller index than that of $j_1$, violating the monotonicity requirement induced by the delivery order. Tree nodes~7--9 are pruned since piece $j_1$ is a shared piece whose inherited placement fixes its bin assignment to $\phi(j_1)$, yet the partial mapping assigns it to bin $i_2$. 
For the sibling-node dominance rule, node~4 and node~6 correspond to two child nodes that assign the later piece $j_3$ to the same bin $i_2$, while their parent mappings differ in the assignment of the preceding piece $j_2$. If node~4 is infeasible at bin $i_2$, then node~6 can be pruned without evaluation because its bin $i_2$ contains a strict superset of the pieces assigned in node~4.
\begin{figure}[ht]
    \centering
    \includegraphics[width=0.9\linewidth]{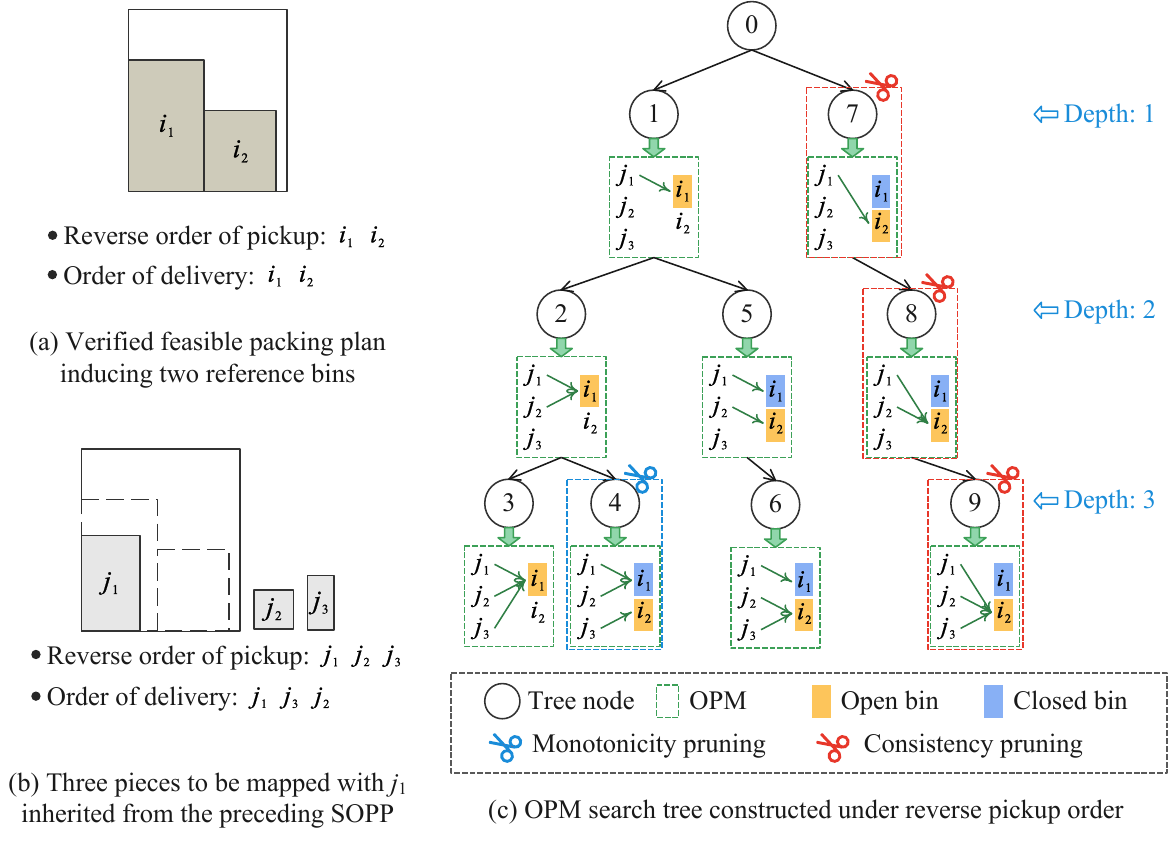}
    \caption{Illustration of the OPM search tree and the associated pruning rules. Panel (a) shows a verified feasible packing plan that induces the reference bins. Panel (b) shows the pieces to be mapped, including one inherited piece whose placement must be preserved from the preceding SOPP. Panel (c) presents the resulting OPM search tree with the resulting pruning outcomes.}
    \label{fig:opm_tree}
\end{figure}

In implementation, pruning rules~\ref{prop:mono_cut} and \ref{prop:cons_cut} can be applied before node expansion by precomputing the admissible bin set for each piece as the intersection of bins allowed by cross-order monotonicity and inherited-placement consistency.

\section{Acceleration Strategies for Dominance Checking}\label{sec:strategies}

As established in Section~\ref{sec:alg_ver_sopp}, packing feasibility is verified at the SOPP level. Dominance screening therefore compares an unverified SOPP with previously verified feasible SOPPs stored in a repository. Although this reuse mechanism can avoid repeated packing verification, an exhaustive comparison with all stored SOPPs may become costly as the repository grows. This section introduces three acceleration strategies: descriptor-based retrieval to identify plausible candidates (Section~\ref{sec:strategies_descriptor}), a candidate cap to limit screening effort (Section~\ref{sec:strategies_threshold}), and hot-store sampling to prioritize recently stored SOPPs (Section~\ref{sec:strategies_hot_store}). Together, these strategies allow dominance screening to serve as a lightweight inference step before computationally demanding exact packing verification is invoked.

\subsection{Descriptor-Based Candidate Retrieval}
\label{sec:strategies_descriptor}

The first acceleration strategy uses aggregate packing descriptors to retrieve a smaller set of candidate reference SOPPs before performing detailed dominance checking. Each verified SOPP is associated with a set of onboard items, and three inexpensive descriptors are computed from this item set: the total item area, the maximum item side length, and the maximum single-item area. 
These descriptors provide necessary conditions for dominance: if any descriptor value of a verified SOPP is smaller than the corresponding value of the unverified SOPP, the verified SOPP cannot dominate the unverified SOPP, because its packing pattern cannot contain all items in the unverified SOPP.
Hence, descriptor-based retrieval first filters the repository and retains only verified SOPPs that are compatible with the unverified SOPP according to these aggregate measures. Detailed dominance checking is then applied only to the retrieved candidates. This strategy does not alter the dominance criterion itself; it only reduces the number of reference SOPPs for which the more expensive order-preserving mapping test is attempted. The precise indexing and retrieval procedure is provided in Electronic Companion.

\subsection{Storage and Candidate Retrieval Thresholds}
\label{sec:strategies_threshold}
Beyond descriptor-based retrieval, two threshold parameters are introduced to control the size of the verified-SOPP repository and the number of retrieved candidates examined in each dominance trial.

To limit the worst-case screening effort for each unverified SOPP, we introduce a candidate cap parameter $\mu^\text{candi}$, which specifies the maximum number of retrieved candidates subjected to detailed dominance checking. Let $\mathcal{C}(\psi)$ denote the set of verified SOPPs retrieved for an unverified SOPP $\psi$ by the descriptor-based rule. If $|\mathcal{C}(\psi)| \le \mu^\text{candi}$, all retrieved candidates are examined. Otherwise, only a subset of size $\mu^\text{candi}$ is selected from $\mathcal{C}(\psi)$ for detailed dominance checking. The selection and ordering of this subset are further guided by the hot-store sampling rule introduced in Section~\ref{sec:strategies_hot_store}.

Not every verified SOPP is equally useful for future screening. If all verified SOPPs are retained, the repository may become populated by weak reference candidates that enlarge retrieval sets without materially improving dominance detection. To mitigate this effect, we use the branch count from the CP solver, which is employed in the final stage of the sequential verification process described in Section \ref{sec:alg_ver_hier}, as a proxy for verification effort. A high branch count suggests that both the heuristic and dominance stages yield inconclusive answers, and that feasibility ultimately requires extensive CP search to verify. Let $\mu^\text{entry}$ denote the repository entry threshold defined on this count. SOPPs requiring more than $\mu^\text{entry}$ explored branches to verify feasibility are stored and serve as reference points for future dominance screening.

Together, the retrieval and storage thresholds reduce the overhead of dominance screening and allow the repository of verified SOPPs to scale without requiring exhaustive comparison at every screening attempt.

\subsection{Hot-Store-Biased Candidate Ordering}
\label{sec:strategies_hot_store}

When sampling from $\mathcal{C}(\psi)$, a purely random selection may be inefficient because it treats all retrieved candidates as equally informative. In an iterative routing algorithm, recently generated routes are often more relevant to subsequent search steps than routes generated much earlier. This is because new routes are typically obtained by modifying the current or recently accepted solution, so SOPPs verified in recent iterations are more likely to share loading structures with newly generated routes. We therefore introduce a hot-store-biased sampling rule that gives higher priority to recently stored SOPPs while retaining randomness in candidate selection.

Let $\widetilde{\mathcal{C}}(\psi) = (\bar\psi_1,\ldots,\bar\psi_{|\widetilde{\mathcal{C}}(\psi)|})$ denote the ordered version of $\mathcal{C}(\psi)$ according to the hot-store priority, so that $|\widetilde{\mathcal{C}}(\psi)|=|\mathcal{C}(\psi)|$. Candidates are arranged from the most recently stored SOPP to the oldest one, and the subscript denotes the position in this ordered pool. When $|\widetilde{\mathcal{C}}(\psi)| > \mu^\text{candi}$, candidates are sampled without replacement as follows. At each draw, let $\mathcal{P}$ denote the current number of remaining candidates in the pool. A variate $u \sim \mathrm{Uniform}(0,1)$ is drawn and the selected index is
\[
q = \lfloor u^{\alpha^\text{hot}} \cdot \mathcal{P} \rfloor + 1,
\]
where $\alpha^\text{hot} \ge 1$ controls the strength of the recency bias. Because $u \in [0,1)$, raising $u$ to the power $\alpha^\text{hot}$ skews the distribution toward zero when $\alpha^\text{hot}>1$, which maps to smaller indices and hence more recently stored SOPPs. The implied selection probability of position $q$ in a pool of size $\mathcal{P}$ is
\[
p(\bar\psi_q) = \left(\frac{q}{\mathcal{P}}\right)^{1/\alpha^\text{hot}} - \left(\frac{q-1}{\mathcal{P}}\right)^{1/\alpha^\text{hot}}, \qquad q=1,\ldots,\mathcal{P}.
\]
When $\alpha^\text{hot}=1$, the rule assigns equal probability to all positions, reducing to uniform sampling over $\widetilde{\mathcal{C}}(\psi)$. Larger values of $\alpha^\text{hot}$ concentrate mass toward recently stored SOPPs. The selected candidates are then examined in increasing order of their positions in $\widetilde{\mathcal{C}}(\psi)$, so that more recently stored SOPPs are checked first. This rule does not alter the dominance criterion; it only affects which retrieved candidates are tested under the candidate cap and in what order.

\section{Numerical Experiments}\label{sec:experiments}
The effectiveness of the proposed dominance-based method is evaluated through numerical experiments based on benchmark instances of the PDP with two-dimensional loading constraints (2L-PDP) \citep{ovgu_2lpdp_instances_2017}. In the experiments, the proposed framework is embedded in a large neighborhood search (LNS) algorithm, which generates routing decisions while the dominance-based module verifies packing feasibility, as detailed in Appendix~\ref{sec:app_lns}. The selected dataset includes 25 requests, a maximum route duration of 300 distance units, and vehicle dimensions $W=40$ and $L=20$. Additional parameters introduced in this work include the vehicle operational cost $c^{\mathrm o}=1$ per unit of distance and the penalty cost $c^{\mathrm p}=500$ per unserved request.
As the proposed method is intended as a generic module applicable to a broad range of routing algorithms, the numerical experiments focus on runtime performance rather than optimality. To isolate the effect of feasibility verification, a relatively large vehicle deployment cost $c^{\mathrm d}=300$ is adopted, thereby reducing routing variability and emphasizing differences induced by feasibility checks.

Algorithmic parameters are set as follows. The maximum number of retrieved candidates for each unverified SOPP is set to $\mu^\text{candi}=30$, and the repository entry threshold is set to $\mu^\text{entry}=60$, whose values are justified in Section~\ref{sec:re_speedup}.
The descriptor-based retrieval procedure uses discretized values of the three aggregate descriptors to organize the repository. Let $\Delta_A$, $\Delta_S$, and $\Delta_M$ denote the discretization widths for the total-area, maximum-side-length, and maximum-single-item-area descriptors, respectively. In the baseline setting, the resulting values are $\Delta_A=160$, $\Delta_S=1$, and $\Delta_M=16$, following the instance-dependent construction described in Electronic Companion. 
The bias factor is set to $\alpha^\text{hot}=2$.

For each experimental setting, the algorithm is run 10 times, and the median is reported to reduce the impact of heuristic variability. The CP-SAT solver is used to solve the CP formulation. All computational experiments are conducted on the DelftBlue supercomputer \citep{DHPC2024}. Each parallel run uses a single CPU core, ensuring a fair comparison of runtimes across instances.

\subsection{Benchmark Comparison and Performance Analysis}\label{sec:re_benchmark}
This section evaluates the computational performance of the proposed dominance-based method by benchmarking it against a standard sequential approach in which the dominance-based verification stage is skipped. In the benchmark approach, feasibility is checked directly using the heuristic packing algorithm, followed by the exact algorithm if necessary. Sections~\ref{sec:re_benchmark_item} and \ref{sec:re_benchmark_reg} examine the effects of item size, item shape, and trunk size scaling on relative computational performance. 
Section~\ref{sec:re_benchmark_burden} investigates the mechanism of computational improvement through verification burden.

\subsubsection{Impact of Item Size and Trunk Size Scaling}\label{sec:re_benchmark_item}
Item sizes and trunk sizes jointly determine the tightness of the packing problem and therefore influence the effectiveness of dominance verification. To examine how these two factors affect the relative computational performance of the proposed method, one instance is taken as the baseline case and systematically perturbed along the item-size and trunk-size dimensions.

Two scaling factors are introduced. Let $(\alpha_{\mathrm{s}}, \alpha_{\mathrm{e}})$ denote the item-size scaling factors, and let $\beta$ denote the trunk-size scaling factor. To vary item sizes without substantially changing the average packing difficulty, the average item area is kept approximately unchanged. Specifically, a proportion $\varsigma$ of items is shrunk by a factor $\alpha_{\mathrm{s}}$, while the remaining proportion $1-\varsigma$ is enlarged by a factor $\alpha_{\mathrm{e}}$, where
\begin{equation}\label{eq:item_scale_factor}
    \varsigma \, {\alpha_{\mathrm{s}}}^2 + (1-\varsigma) \, {\alpha_{\mathrm{e}}}^2 = 1.
\end{equation}
In this experiment, shrinking and enlarging are applied with equal probability, i.e., $\varsigma = 0.5$. The shrinking factor $\alpha_{\mathrm{s}}$ ranges from $0.1$ to $1.0$, and the enlarging factor $\alpha_{\mathrm{e}}$ is derived from Eq.~\eqref{eq:item_scale_factor}. When $\alpha_{\mathrm{s}} = 1$, both factors equal one and all item sizes are preserved; a smaller $\alpha_{\mathrm{s}}$ implies a larger $\alpha_{\mathrm{e}}$, reflecting greater heterogeneity in item sizes. The trunk-size scaling factor $\beta$ ranges from $0.6$ to $1.0$ and is applied to the trunk width.

Figure~\ref{fig:scale_item} reports the median objective and runtime improvements over the benchmark for each parameter combination. Statistical significance within each cell is assessed using a one-sided paired Wilcoxon signed-rank test at the 0.05 level, which is appropriate for matched observations when the distribution of paired differences is not assumed to be normal \citep{wilcoxon1945individual,conover1999practical}.
\begin{figure}[htbp]
    \centering
    \begin{subfigure}[b]{0.48\linewidth}
        \centering
        \includegraphics[width=\linewidth]{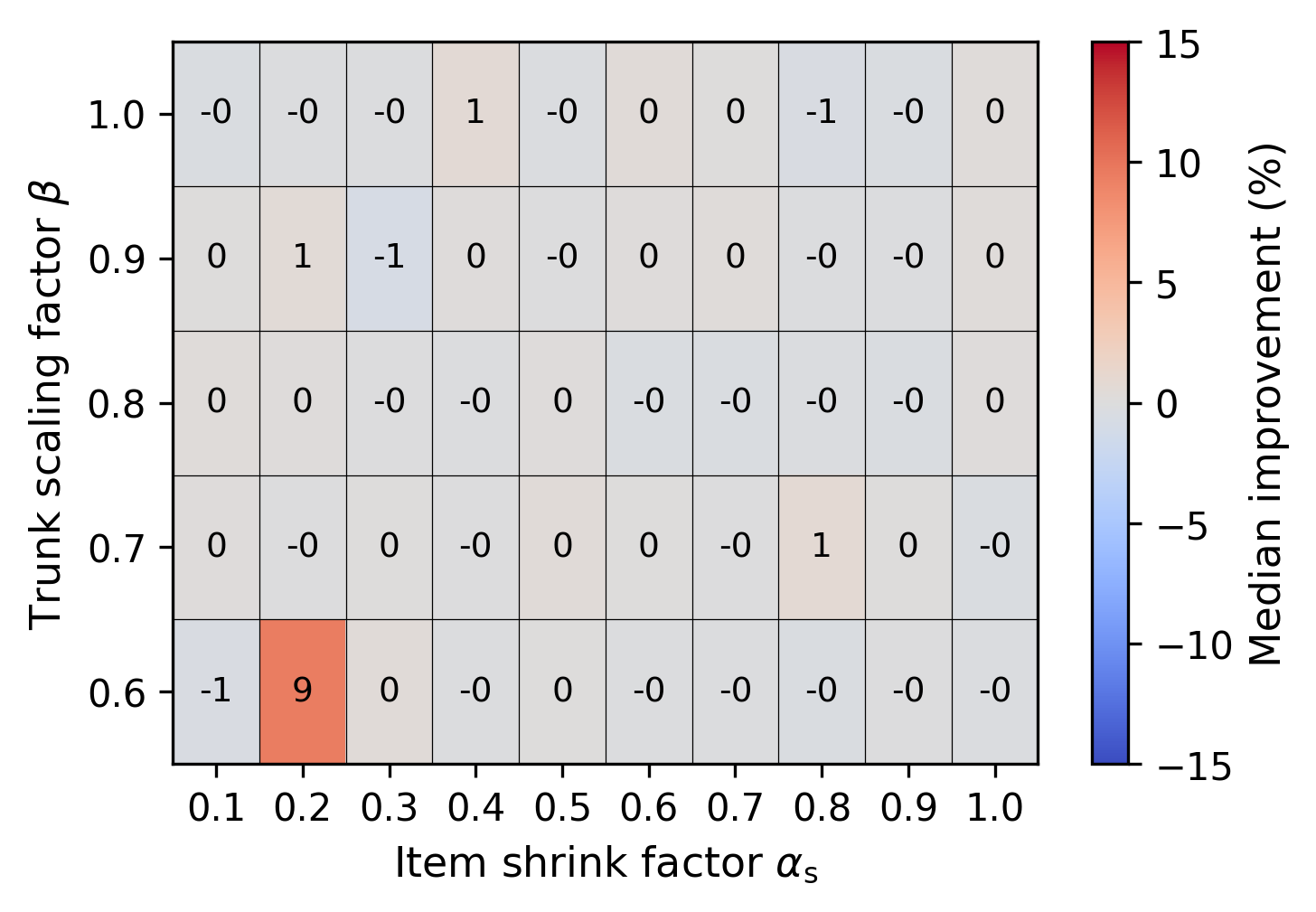}
        \caption{Objective}
        \label{fig:scale_item_obj}
    \end{subfigure}
    \hfill
    \begin{subfigure}[b]{0.48\linewidth}
        \centering
        \includegraphics[width=\linewidth]{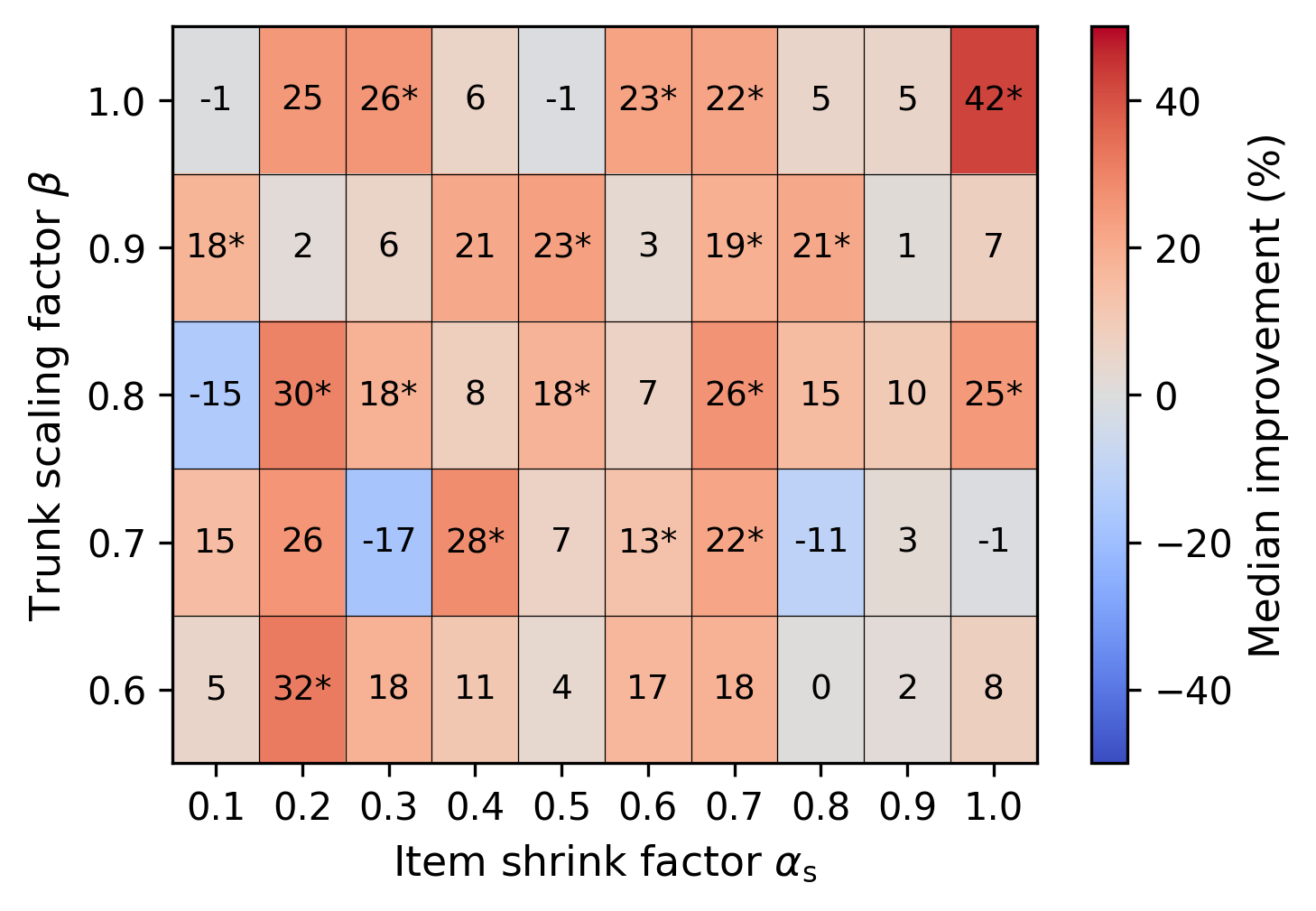}
        \caption{Runtime}
        \label{fig:scale_item_time}
    \end{subfigure}
    \caption{Median improvement over the benchmark across item-size and trunk-width scaling. Positive values indicate better performance of the proposed method, corresponding to a smaller objective value in (a) and a smaller runtime in (b). An asterisk $^*$ denotes statistical significance at the 0.05 level based on a one-sided paired Wilcoxon signed-rank test.}
    \label{fig:scale_item}
\end{figure}

Figure~\ref{fig:scale_item_obj} shows negligible objective differences between the two approaches, confirming that the proposed dominance rule primarily acts as a feasibility screening technique with little impact on solution quality. 
Figure~\ref{fig:scale_item_time} further shows that the proposed method generally reduces runtime across most combinations of the item shrink factor $\alpha_{\mathrm{s}}$ and trunk scaling factor $\beta$. The median runtime improvement is positive in the majority of tested settings, with several cases achieving improvements above 20\% and a maximum improvement of 42\%.

Overall, these results show that the proposed dominance-based screening can substantially accelerate feasibility checking without deteriorating solution quality. At the same time, the magnitude of the runtime improvement is instance-dependent, indicating that the benefit of dominance transfer is shaped by the detailed structure of the generated routes and packing configurations rather than by item or trunk scaling alone.

\subsubsection{Impact of Item Shape and Trunk Size Scaling}\label{sec:re_benchmark_reg}
While item sizes determine the overall tightness of the packing problem, item shapes influence the geometric flexibility available for transferring feasibility across SOPPs. This subsection therefore examines how item-shape regularization affects the relative computational performance of the proposed method under different trunk-size levels.

In addition to the trunk-size scaling factor $\beta$, a regularization ratio $\gamma$ is introduced to modify item shapes while keeping item areas approximately unchanged. For an item $i$ with width $w_i$ and height $h_i$, let $\delta_i$ and $\delta_i'$ denote its original and modified aspect ratios, respectively:
\[
\delta_i = \frac{\max(w_i,h_i)}{\min(w_i,h_i)} \geq 1, \qquad
\delta_i' = 1 + (1-\gamma)(\delta_i-1).
\]
The modified width $w_i'$ and height $h_i'$ are then computed as
\[
w_i' = \left\lceil \sqrt{w_i h_i \delta_i'} \right\rceil, \qquad
h_i' = \left\lceil \sqrt{\frac{w_i h_i}{\delta_i'}} \right\rceil.
\]
Hence, $\gamma=0$ preserves the original item shape, whereas $\gamma=1$ regularizes all items to square-like shapes. In the experiment, $\gamma$ ranges from $0$ to $1$, and $\beta$ ranges from $0.6$ to $1.0$.
\begin{figure}[htbp]
    \centering
    \begin{subfigure}[b]{0.48\linewidth}
        \centering
        \includegraphics[width=\linewidth]{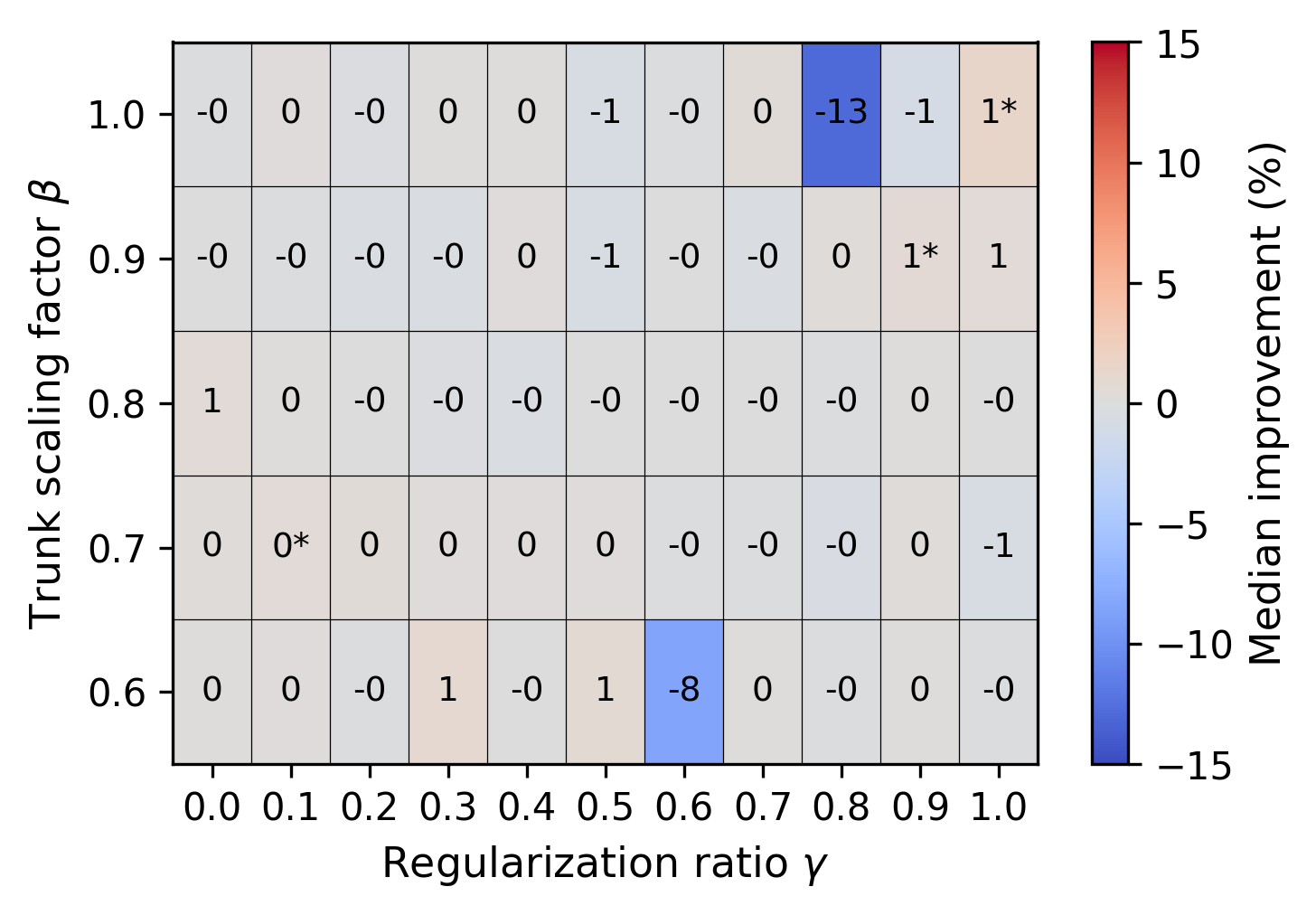}
        \caption{Objective}
        \label{fig:reg_item_obj}
    \end{subfigure}
    \hfill
    \begin{subfigure}[b]{0.48\linewidth}
        \centering
        \includegraphics[width=\linewidth]{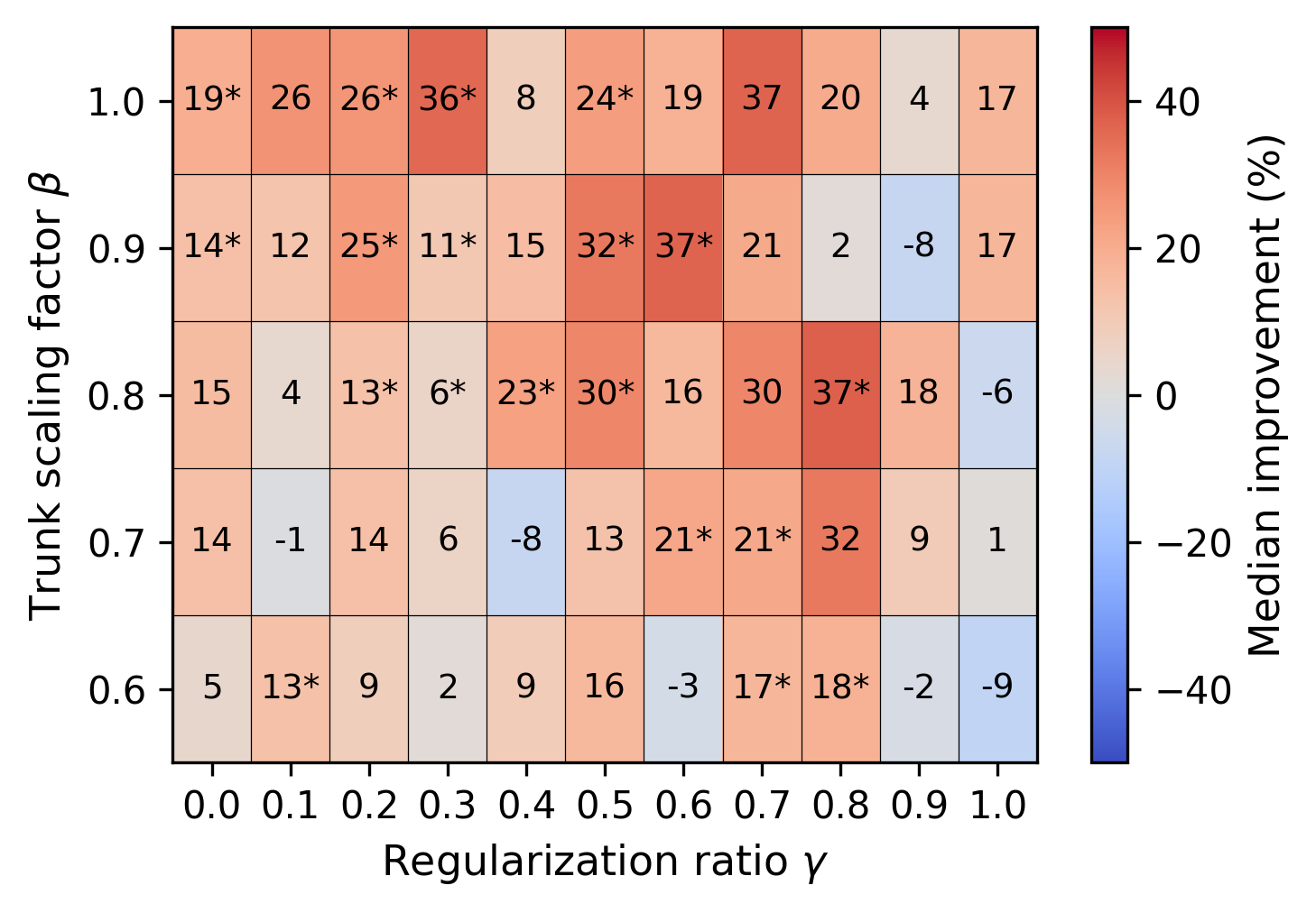}
        \caption{Runtime}
        \label{fig:reg_item_time}
    \end{subfigure}
    \caption{Median improvement over the benchmark across item shape regularization ratio and trunk-width scaling. Positive values indicate better performance of the proposed method, corresponding to a smaller objective value in (a) and a smaller runtime in (b). An asterisk $^*$ denotes statistical significance at the 0.05 level based on a one-sided paired Wilcoxon signed-rank test.}
    \label{fig:reg_item}
\end{figure}

Figure~\ref{fig:reg_item} reports the median objective and runtime improvements over the benchmark. 
Figure~\ref{fig:reg_item_obj} again shows negligible objective differences, consistent with the role of the proposed dominance rule as a feasibility screening technique.
Figure~\ref{fig:reg_item_time} shows that achieves positive median runtime improvement over most combinations of regularization ratio and trunk scaling factor, with several improvements exceeding 30\%. The largest improvement of 37\% is observed at $\gamma=$ 0.6--0.9. This broad improvement suggests that the method is not restricted to a narrow class of instances.

The runtime improvements are most pronounced at moderate item-regularization levels. In this range, item shapes retain some geometric diversity while avoiding highly irregular configurations. This balance can support dominance verification in two ways: geometric diversity provides more opportunities for matching items into previously verified packing layouts, while moderate regularity helps preserve structural similarity across SOPPs. As a result, feasible packing patterns are more likely to be transferable from verified SOPPs to unverified ones.
The results also indicate that larger trunk scaling factors can further strengthen the runtime benefit. Larger trunks provide more spatial flexibility when embedding items from an unverified SOPP into a previously verified packing plan, increasing the chance that dominance can certify feasibility before the algorithm falls back to exact checking.

Overall, these results show that the proposed dominance-based screening is effective across a broad range of geometric settings. Its strongest benefits arise when item shapes are moderately regularized and the trunk provides sufficient spatial flexibility, allowing feasible packing structures to be transferred more effectively across SOPPs.

\subsubsection{Mechanism of Computational Improvement}\label{sec:re_benchmark_burden}
To further examine the mechanism behind the computational gains of the proposed approach, this section analyzes the trade-off between the additional cost of dominance screening and the exact-verification effort it can avoid. Although dominance screening introduces additional verification overhead, successful dominance checks can prevent SOPPs from reaching the substantially more expensive exact-verification stage. To quantify this trade-off, we introduce a run-level measure that summarizes the average verification cost required by the sequential verification process. For a given run $u$, let $p_u^{\mathrm{H}}$, $p_u^{\mathrm{D}}$, and $p_u^{\mathrm{C}}$ denote the proportions of SOPPs whose feasibility is verified at the heuristic, dominance, and CP stages, respectively. By construction, 
$p_u^{\mathrm{H}} + p_u^{\mathrm{D}} + p_u^{\mathrm{C}} = 1$.

To reflect the computational effort associated with this sequential pipeline, we define a time-weighted verification cost. Let $\tau_u^{\mathrm{H}}$, $\tau_u^{\mathrm{D}}$, and $\tau_u^{\mathrm{C}}$ denote the median cumulative verification time of SOPPs that terminate at the heuristic, dominance, and CP stages in run $u$, respectively. The time-weighted verification cost is then defined as
\[
    B_u = \tau_u^{\mathrm{H}} p_u^{\mathrm{H}} 
        + \tau_u^{\mathrm{D}} p_u^{\mathrm{D}} 
        + \tau_u^{\mathrm{C}} p_u^{\mathrm{C}} .
\]
A smaller value of $B_u$ indicates that, on average, SOPPs are verified with lower cumulative verification cost in the proposed pipeline, either because more SOPPs terminate before the expensive CP stage or because the stages they pass through require less time.

To relate this cost measure to computational performance, we define a binary run-level indicator equal to one if the proposed method is faster than the benchmark and zero otherwise. The analysis is based on the pooled solutions from the size-scaling experiments in Section~\ref{sec:re_benchmark_item} and the shape-regularization experiments in Section~\ref{sec:re_benchmark_reg}.
Figure~\ref{fig:dom_outperformance_vs_burden} shows the empirical outperformance probability, computed as the average of this binary indicator, across different levels of the time-weighted verification cost, together with 95\% Wilson confidence intervals and the number of runs in each bin.

\begin{figure}[htbp]
    \centering
    \begin{subfigure}[b]{0.48\linewidth}
        \centering
        \includegraphics[width=\linewidth]{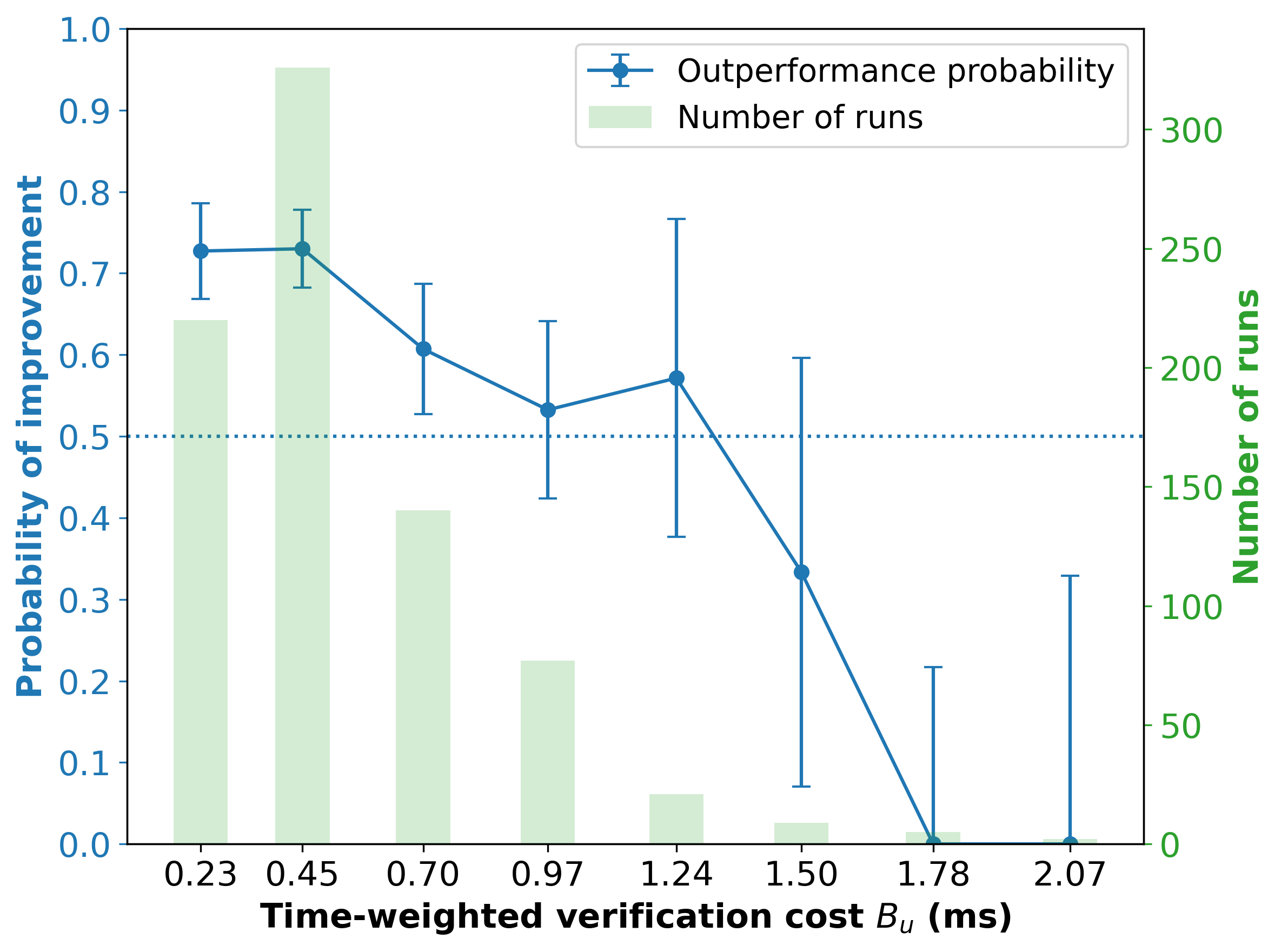}
        \caption{Outperformance probability}
        \label{fig:dom_outperformance_vs_burden}
    \end{subfigure}
    \hfill
    \begin{subfigure}[b]{0.48\linewidth}
        \centering
        \includegraphics[width=\linewidth]{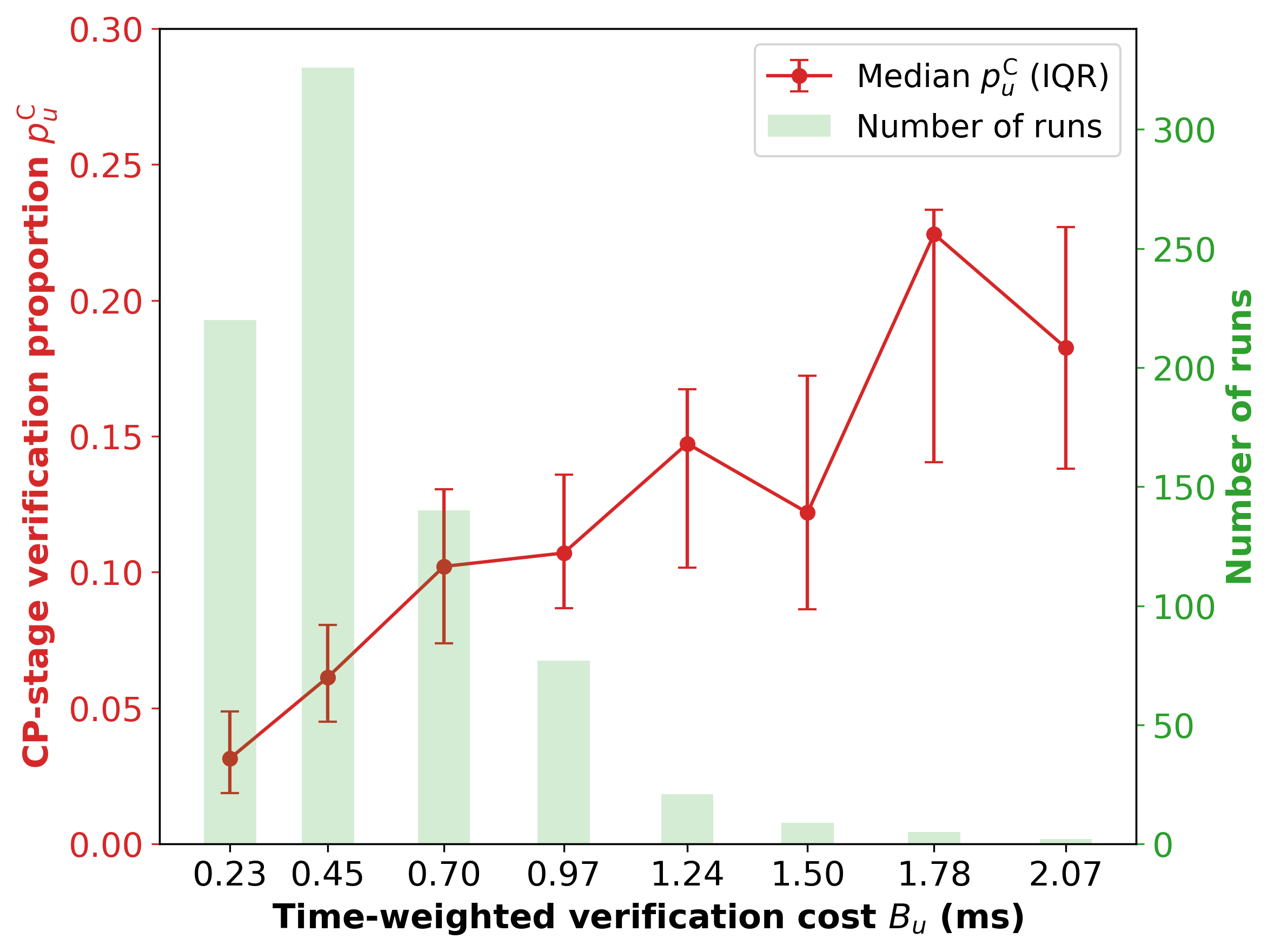}
        \caption{CP-stage verification proportion}
        \label{fig:dom_outperformance_vs_cp}
    \end{subfigure}
\caption{Outperformance probability and CP-stage verification proportion across time-weighted verification-cost bins.}
\end{figure}

The results show that the proposed method outperforms the benchmark in most bins of the verification-cost measure, indicating that the dominance stage is generally beneficial across a broad range of verification conditions. The outperformance probability is particularly high in the low-cost region, where dominance screening successfully verifies a substantial share of SOPPs, thereby avoiding more expensive CP calls. In these cases, the reduction in CP effort outweighs the cost of the additional dominance checks, yielding a net computational gain.
Figure~\ref{fig:dom_outperformance_vs_cp} complements the outperformance analysis by showing how residual exact verification changes across the verification-cost range. 
The median CP-stage verification proportion $p_u^{\mathrm{C}}$ increases from the low- to high-cost bins, indicating that larger values of $B_u$ are associated with a greater share of SOPPs reaching the CP stage.
This provides mechanism-level support for Figure~\ref{fig:dom_outperformance_vs_burden}: the reduction in outperformance probability at higher values of $B_u$ is linked to weaker filtering by the preceding verification stages. The sparsely populated high-cost bins mainly characterize the upper tail of the $B_u$ distribution.

As the time-weighted verification cost increases, the outperformance probability tends to decrease, reflecting the opposite side of the same trade-off. When dominance screening verifies few SOPPs or incurs high per-check cost, $B_u$ increases: more SOPPs proceed to CP, or the dominance stage itself contributes substantially to verification time. In such cases, the proposed method pays the dominance-screening cost but obtains limited CP savings, so the additional stage may weaken or eliminate the runtime advantage over the benchmark. Nevertheless, most runs fall in bins where the proposed method is more likely to outperform the benchmark, with wider confidence intervals mainly appearing in the sparsely populated high-cost region.

Overall, the analysis shows that the effectiveness of the proposed framework depends on the balance between dominance-screening overhead and avoided CP effort, with the strongest gains concentrated in runs where dominance screening replaces a meaningful share of exact feasibility checks.

\subsection{Effect of Speed-up Techniques}\label{sec:re_speedup}

To evaluate the computational impact of the speed-up techniques proposed in Sections~\ref{sec:strategies_threshold} and~\ref{sec:strategies_hot_store}, this section examines how the hot-store-biased sampling rule and the threshold-based storage and retrieval rules affect the balance between dominance-reuse effectiveness and screening overhead. We first isolate the effect of the hot-store-biased sampling rule through a paired comparison between the proposed method with and without this rule. Pooling the two experiment series in Sections~\ref{sec:re_benchmark_item} and~\ref{sec:re_benchmark_reg}, the hot-store-biased sampling rule improves runtime by a median of 11.72\% and outperforms the non-hot-store variant in 63.4\% of the paired runs. The runtime reduction is statistically significant under the one-sided paired Wilcoxon signed-rank test at the 5\% significance level. The objective values remain essentially unchanged, indicating that the gain is mainly computational.

We next examine the candidate retrieval size $\mu^{\text{candi}}$ and the repository entry threshold $\mu^{\text{entry}}$. This experiment uses the baseline instance setting and varies only the two speed-up parameters, with both $\mu^{\text{candi}}$ and $\mu^{\text{entry}}$ ranging from 5 to 150. Figure~\ref{fig:hmp_grid_obj} reports the median and interquartile range (IQR) of the objective values over the tested parameter grid, while Figure~\ref{fig:hmp_grid_time} presents the corresponding statistics for runtime.

\begin{figure}[ht]
    \centering
    \includegraphics[width=0.9\linewidth]{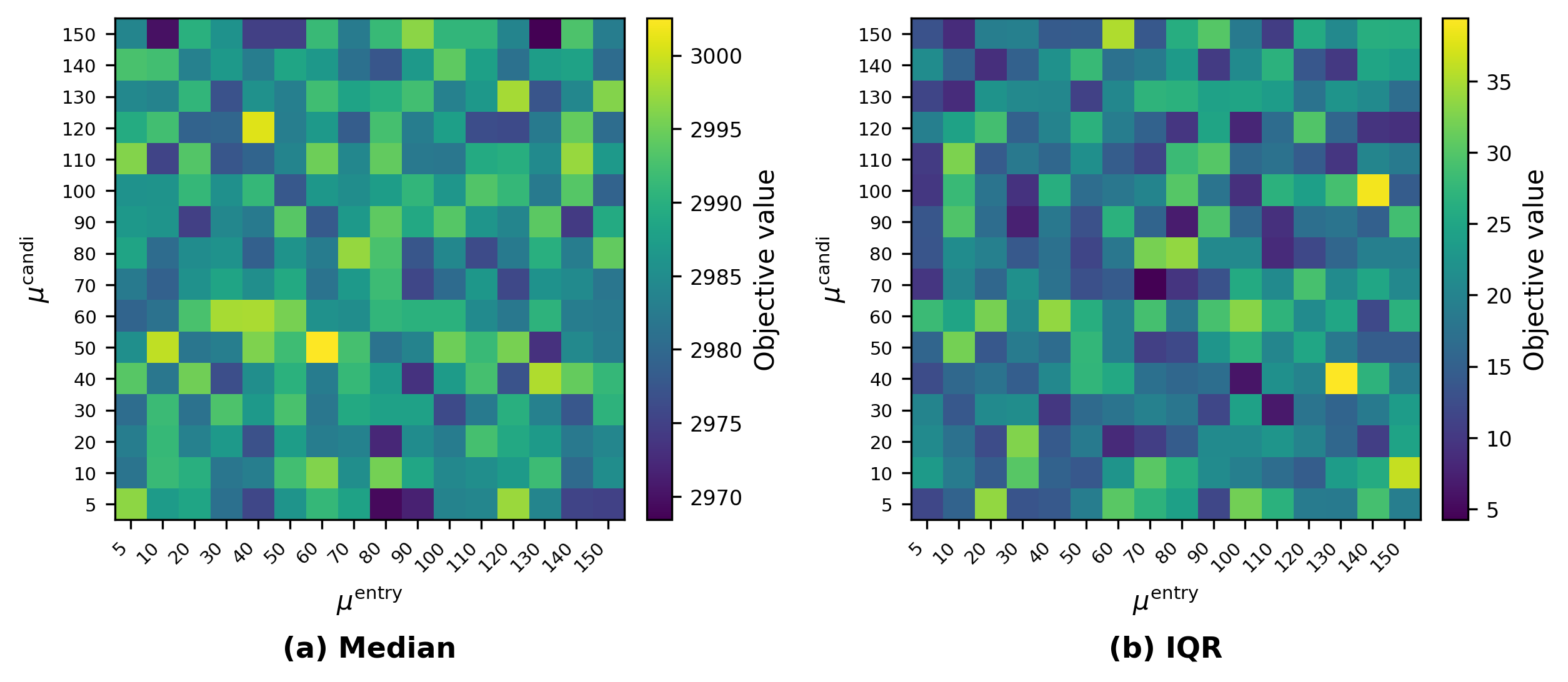}
    \caption{Median and IQR of the objective value over the $(\mu^\text{candi}, \mu^\text{entry})$ grid}
    \label{fig:hmp_grid_obj}
\end{figure}
\begin{figure}[ht]
    \centering
    \includegraphics[width=0.9\linewidth]{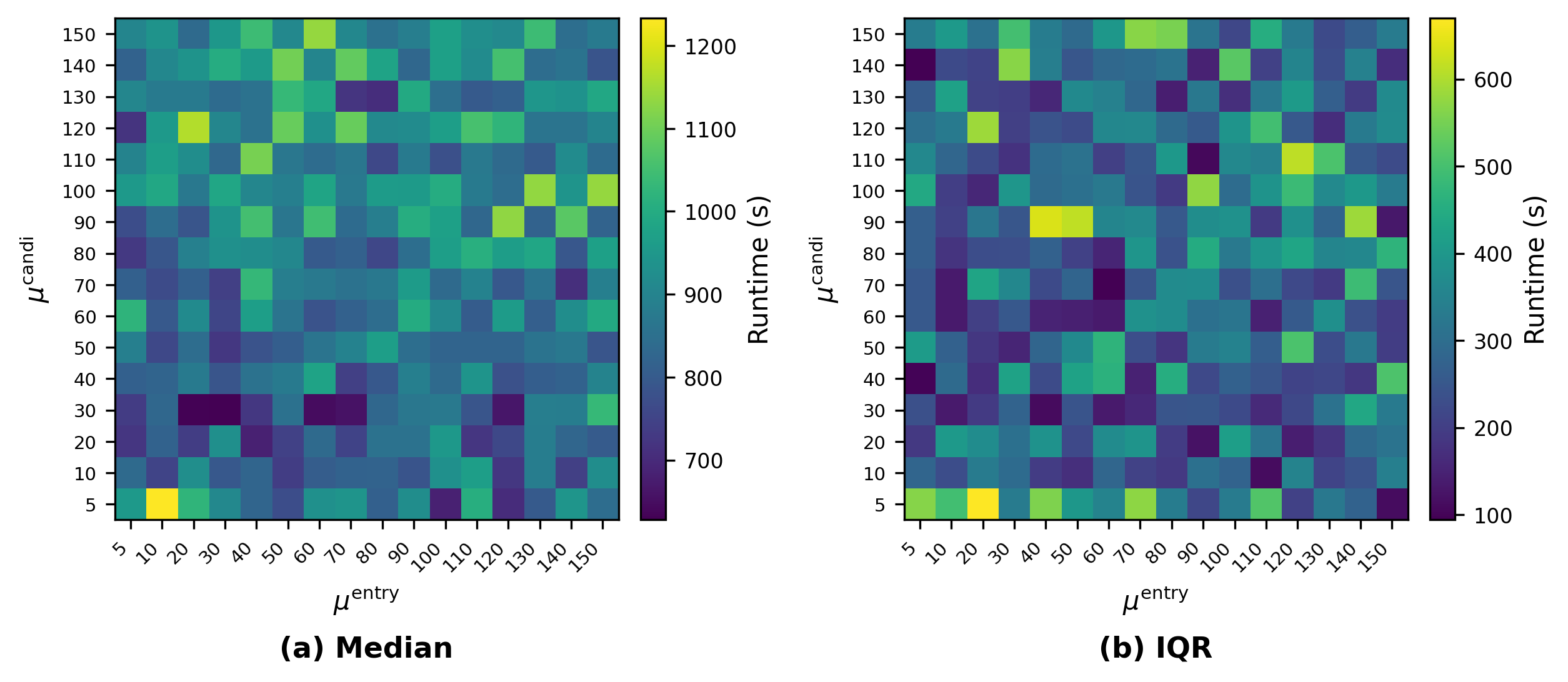}
    \caption{Median and IQR of runtime over the $(\mu^\text{candi}, \mu^\text{entry})$ grid}
    \label{fig:hmp_grid_time}
\end{figure}

Figure~\ref{fig:hmp_grid_obj} shows that the objective values remain highly stable across the entire parameter grid. The median objective value ranges only from approximately 2968 to 3003, confirming that $\mu^{\text{candi}}$ and $\mu^{\text{entry}}$ mainly affect the computational effort of packing-feasibility screening, while having negligible influence on the solution quality determined by the routing algorithm.

By contrast, Figure~\ref{fig:hmp_grid_time} shows that runtime is substantially affected by the two threshold parameters. The median runtime varies from approximately 628 seconds to 1233 seconds across the tested grid, indicating that parameter selection has a clear impact on computational efficiency. 
The lowest median runtime is obtained at $(\mu^{\text{candi}},\mu^{\text{entry}})=(30,30)$, with a median runtime of 628 seconds. The setting $(30,60)$ achieves a comparable median runtime of 648 seconds, while reducing the runtime IQR from 277 to 136 seconds, corresponding to a reduction of about 51.0\%. This indicates better stability across runs and supports the use of $(\mu^{\text{candi}},\mu^{\text{entry}})=(30,60)$ as the baseline setting. More broadly, the concentration of low median runtimes in this intermediate region of the grid suggests that overly restrictive or overly permissive threshold settings can reduce computational efficiency.

The candidate retrieval size $\mu^{\text{candi}}$ has a particularly direct impact on runtime. Averaged over all values of $\mu^{\text{entry}}$, the lowest median runtime is obtained when $\mu^{\text{candi}}=30$, followed by $\mu^{\text{candi}}=20$ and $\mu^{\text{candi}}=10$. This pattern is consistent with the role of $\mu^{\text{candi}}$: when it is too small, useful reference SOPPs may be missed, reducing the chance that dominance can verify feasibility before the exact check is needed; when it is too large, the dominance procedure may examine many candidates for each SOPP, increasing the worst-case screening effort. 

The repository entry threshold $\mu^{\text{entry}}$ further affects runtime by shaping the composition of the repository. In the tested grid, low entry thresholds such as $\mu^{\text{entry}}\le 20$ create a broader repository but may introduce more weak or redundant reference candidates, whereas high thresholds such as $\mu^{\text{entry}}\ge 90$ make the repository more selective but may exclude reusable SOPPs. This repository composition interacts with $\mu^{\text{candi}}$: an overly selective repository combined with a small candidate cap may make dominance reuse depend heavily on whether the sampled candidates happen to include a structurally compatible reference, while a broad repository combined with a large candidate cap may increase screening overhead without a proportional gain in dominance success. Therefore, the two parameters jointly control the balance between reuse opportunities and screening overhead: $\mu^{\text{entry}}$ determines what is stored, while $\mu^{\text{candi}}$ determines how much dominance screening is attempted.

Figure~\ref{fig:hmp_grid_time}b further shows that runtime variability is also sensitive to the parameter setting, suggesting that inappropriate threshold choices can make the computational performance more dependent on the specific SOPPs generated in each run. In contrast, moderate values of the two parameters generally provide a better balance between median runtime and stability.

Overall, the sensitivity analysis suggests that the two thresholds mainly influence computational efficiency rather than solution quality. For the tested instances, avoiding extreme values of $\mu^{\text{candi}}$ is especially important because it directly controls the worst-case number of dominance trials.

\section{Conclusions}\label{sec:conclusions}
This study addresses the computational burden of repeated packing-feasibility verification in packing-constrained routing problems by introducing a dominance-based framework that reuses previously verified packing states to infer the feasibility of new ones. By leveraging structural relationships among route-dependent packing configurations, the framework reduces the reliance on repeated exact feasibility checks and transforms packing verification into a reusable inference process.

Computational results demonstrate that the proposed dominance-based screening method substantially accelerates feasibility verification across a broad class of instances, achieving runtime reductions of up to 42\% without affecting solution quality. The performance gain is primarily driven by the reduction in exact-verification calls, confirming the effectiveness of dominance-based inference as a computational screening mechanism.
The mechanism analysis further shows that the value of dominance screening depends on the trade-off between its additional verification cost and the exact-verification effort it can eliminate.
The proposed method is most effective when dominance screening is sufficiently informative to divert a substantial fraction of SOPPs away from the most computationally expensive stage, making the added screening cost worthwhile.
The sensitivity analysis further indicates that more stable runtime performance is achieved when the entry and retrieval thresholds are kept within empirically favorable ranges. The entry threshold shapes the informativeness and selectivity of the repository, while the retrieval threshold controls the extent to which stored structures are exploited without excessive dominance-screening overhead. Additionally, the hot-store-biased sampling rule provides a further computational gain by prioritizing recently verified structures, reducing median runtime by 11.72\%.

Several directions for future research emerge from this work. Future research could investigate how reusable packing-feasibility information can be integrated into the routing search itself, so that dominance relations not only screen candidate routes but also guide the generation of packing-compatible route structures. 
Another direction is to study cross-instance reuse of verified packing states, where repositories constructed from previously solved instances may support feasibility screening for new instances with similar item and route structures. This idea could be combined with learning-based methods, including reinforcement learning, to adaptively prioritize stored states or guide routing decisions based on past verification outcomes. 
The framework may also be extended to richer packing environments, including three-dimensional loading, weight distribution, stacking, fragility, multi-compartment vehicles, and alternative accessibility rules, where the transferability of feasible packing structures becomes more complex. 
Finally, stronger dominance conditions could be developed to characterize classes of packing-constrained routing instances in which feasibility transfer can be verified more tightly or even exactly.

%\clearpage
\appendix
\titleformat{\section}{\normalfont\Large\bfseries}
  {Appendix \thesection}{1em}{}

\section{Exact packing feasibility check}\label{sec:app_exa}

\renewcommand{\thefigure}{A\arabic{figure}}
\setcounter{figure}{0}
\renewcommand{\thetable}{A\arabic{table}}
\setcounter{table}{0}
\renewcommand{\theequation}{A\arabic{equation}}
\setcounter{equation}{0}

Packing feasibility is verified exactly via a constraint programming (CP) formulation. CP approaches have been shown to be effective for solving the two-dimensional orthogonal packing problem (2D-OPP) and its variants \citep{pisinger2007using, clautiaux2008new}. Our formulation extends 2D-OPP by incorporating sequence-dependent no-relocation constraints. Since the purpose is only to determine whether a feasible placement exists, no objective function is introduced.
Given a bin $i$ with size $(W_i,L_i)$ and the set of pieces assigned to it, where the bin may denote either the vehicle trunk or an in-bin packing region arising in dominance verification, the formulation checks whether a feasible placement exists under Conditions~(P1)--(P3) introduced in Section~\ref{sec:problem}. The notation used below is summarized in Table~\ref{tab:notation_opp}.

\textbf{(1) Bin boundary and rotation.}
Constraints \eqref{eq:cp_boundary_x}--\eqref{eq:cp_rot_l} enforce Conditions~(P1) and~(P2) by requiring each piece to lie entirely within the bin while accounting for rotation.
\begin{align}
x_j + \bar{w}_j &\le W_i && \forall j \in J, \label{eq:cp_boundary_x}\\
y_j + \bar{l}_j &\le L_i && \forall j \in J, \label{eq:cp_boundary_y}\\
\bar{w}_j &= (1-\xi_j) w_j + \xi_j l_j && \forall j \in J, \label{eq:cp_rot_w}\\
\bar{l}_j &= (1-\xi_j) l_j + \xi_j w_j && \forall j \in J. \label{eq:cp_rot_l}
\end{align}

\textbf{(2) Non-overlap.}
Constraint \eqref{eq:cp_nonoverlap} enforces Condition~(P1) by requiring every pair of distinct pieces to be separated horizontally or vertically.
\begin{align}
x_j + \bar{w}_j \le x_{j'} \;\vee\;
x_{j'} + \bar{w}_{j'} \le x_j \;\vee\;
y_j + \bar{l}_j \le y_{j'} \;\vee\;
y_{j'} + \bar{l}_{j'} \le y_j
\quad \forall j,j' \in J,\; j \neq j'.\label{eq:cp_nonoverlap}
\end{align}

\textbf{(3) No-relocation constraints.}
Constraint \eqref{eq:cp_nr} enforces the operation-accessibility requirement in Condition (P3a). For each precedence pair $(i,i') \in \theta^\text{p} \cup \theta^\text{d}$, piece $i$ must remain rear-accessible relative to piece $i'$.
\begin{align}
x_j + \bar{w}_j \le x_{j'} \;\vee\;
x_{j'} + \bar{w}_{j'} \le x_j \;\vee\;
y_j \ge y_{j'} + \bar{l}_{j'}
\quad \forall (j,j') \in \theta^\text{p} \cup \theta^\text{d}. \label{eq:cp_nr}
\end{align}

\textbf{(4) Fixed placements.}
Constraint \eqref{eq:cp_fix_pos} enforces the consistent-placement requirement in Condition (P3b) by fixing the positions and orientations of pieces inherited from predecessor SOPPs, in accordance with Proposition~\ref{prop:route_feasibility}.
\begin{align}
x_j = \hat{x}_j,\; y_j = \hat{y}_j,\; \xi_j = \hat{\xi}_j 
\quad \forall j \in J^{\text{fix}}.\label{eq:cp_fix_pos}
\end{align}

% \textbf{(5) Domains.}
% Constraint \eqref{eq:cp_domain} defines the variable domains.
% \begin{align}
% x_i, y_i \ge 0,\; \xi_i \in \{0,1\} \quad \forall i \in I.\label{eq:cp_domain}
% \end{align}

\section{Heuristic packing feasibility check}\label{sec:app_heu}
\renewcommand{\thefigure}{B\arabic{figure}}
\setcounter{figure}{0}
\setcounter{algorithm}{0}
\renewcommand{\thealgorithm}{B\arabic{algorithm}}

This section describes a conservative maximum-open-space (MOS)-based heuristic, summarized in Algorithm~\ref{alg:mos_heuristic}, used as a feasibility oracle within the hierarchical checking scheme. For a given SOPP, items are placed sequentially using a first-fit rule on a set of MOSs, following the principles in \citet{wei2018simulated}.

\textbf{Inherited placements and consistency.}
Following Proposition~\ref{prop:route_feasibility}, when verifying an SOPP with predecessors, placements of items shared with preceding SOPPs are first inherited and treated as fixed. The heuristic then attempts to place the remaining items without any repositioning.

\textbf{MOS representation and placement rule.}
An MOS is a rectangular free region represented by its bottom-left corner and dimensions, $M=(x,y,w,h)$. Let $\mathcal{M}$ denote the current set of MOSs. For each unfixed item $j$, the heuristic evaluates all current MOSs under both orientations, represented by the rotation indicator $\xi\in\{0,1\}$, where $\xi=0$ denotes the original orientation and $\xi=1$ denotes a $90^\circ$ rotation, consistent with the CP formulation. A candidate placement is feasible if the item fits within the MOS and satisfies Condition~(P3a). Let $\placement_j$ denote the set of feasible candidate placements for item $j$. Among feasible candidates, the heuristic applies a bottom-left rule: it selects the candidate with smallest rear coordinate $y$ and breaks ties by smallest lateral coordinate $x$.

\textbf{MOS update.}
After placing an item, the MOS list is updated by (i) removing the used MOS and generating residual MOSs, (ii) clipping any MOS that overlaps the newly placed rectangle and generating the remaining subspaces, and (iii) applying an $x$-aligned dominance filter, under which an MOS is discarded if it is contained in another MOS with the same $x$-coordinate. Figure~\ref{fig:mos_evolution} illustrates the evolution of the MOS set.

\begin{algorithm}[ht]
\caption{MOS-based heuristic packing for sequential SOPP verification}
\label{alg:mos_heuristic}
\begin{algorithmic}[1]
\State \textbf{Input:} Bin $(W_i,L_i)$; ordered collection of SOPPs $\Psi_\route$ associated with route $\route$
\For{each SOPP $\psi \in \Psi_\route$ in ascending SOPP index}
    \State Fix placements of shared items in $J^{\mathrm{fix}}(\psi)$ inherited from predecessor SOPPs, if any
    \State Initialize the MOS set $\mathcal{M}$ by carving the bin free space around the fixed rectangles;
           if no rectangles are fixed, set $\mathcal{M} \gets \{(0,0,W_i,L_i)\}$
    \For{each item $j \in J(\psi)\setminus J^{\mathrm{fix}}(\psi)$ in the pickup order}
        \State $\placement_j \gets \emptyset$
        \For{each $M \in \mathcal{M}$ and $\xi \in \{0,1\}$}
            \If{item $j$ fits in $M$ under rotation $\xi$ and NR-p and NR-d constraints hold}
                \State $\placement_j \gets \placement_j \cup \{(M,\xi)\}$
            \EndIf
        \EndFor
        \If{$\placement_j=\emptyset$}
            \State \Return infeasible
        \EndIf
        \State Select $(M^\star,\xi^\star)\in\placement_j$ by the bottom-left rule
        \State Place item $i$ at $(M^\star,\xi^\star)$; update the MOS set $\mathcal{M}$
    \EndFor
\EndFor
\State \Return feasible and a packing plan
\end{algorithmic}
\end{algorithm}

\begin{figure}[htbp]
    \centering
    \includegraphics[width=0.8\linewidth]{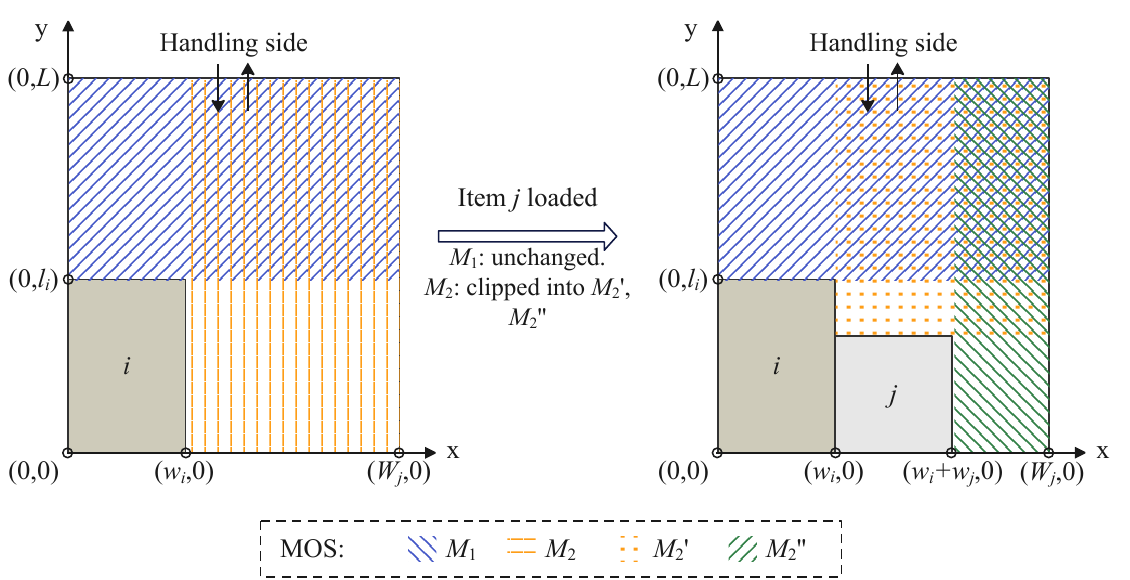}
    \caption{Evolution of MOSs after an item is loaded.}
    \label{fig:mos_evolution}
\end{figure}

\section{Iterative algorithm framework}
\label{sec:app_lns}
A large neighborhood search (LNS) with a simulated annealing (SA) acceptance criterion is used to obtain the routing decisions. The algorithm is run for $\mathcal{I}^\mathrm{max}=2000$ iterations, with initial and final temperatures $\mathcal{T}^\mathrm{start}=100$ and $\mathcal{T}^\mathrm{end}=1$, and a cooling factor $\alpha^\text{cool}=0.9972$. Five destroy and three repair operators are employed.

Two Shaw removal operators are designed to remove $n^\text{remove}_\text{shaw}=0.2|R|$ requests per iteration using a biased random mechanism controlled by $\beta=2$ \citep{ropke2006adaptive}. The first Shaw removal $\phi^\text{d}_1$ uses an overall relatedness measure weighted by $\boldsymbol{\omega}^\text{shaw}=[2,5,9]$, capturing item size, maximum side, and spatial proximity, while $\phi^\text{d}_2$ considers spatial proximity only. 
Three worst removal operators are designed: $\phi^\text{d}_3$ and $\phi^\text{d}_4$ remove $n^\text{remove}_\text{worst}$ worst requests, selected from all routes and iteratively from randomly chosen routes, respectively, while $\phi^\text{d}_5$ removes the worst route.

One best insertion and two regret insertion operators are used, each considering at most $\lambda_2=4$ routes and $\lambda_1=4$ insertion positions. Operator $\phi^\text{r}_1$ selects the position with the least additional cost, while regret operators $\phi^\text{r}_2$ and $\phi^\text{r}_3$ are based on the difference between the globally best and second-best insertion positions and on the difference between the best insertion positions across routes, respectively.

\bibliographystyle{plainnat}
\bibliography{cas-refs}

\vspace{3em}

\begin{table}[ht]
\caption{Notation for the OPP model}
\label{tab:notation_opp}
\small
\centering
\begin{tabularx}{\textwidth}{@{}lX@{}}
\toprule
Notation & Description\\
\midrule
\multicolumn{2}{@{}l}{Sets and Parameters}\\
\midrule
$J$ & Set of pieces to be packed, indexed by $j$\\
$J^{\text{fix}}$ & Set of shared pieces inherited from the previous SOPPs\\
$W_i,\, L_i$ & Width and length of bin $i$\\
$w_j,\, l_j$ & Original width and length of piece $j \in J$\\
$\hat{x}_j,\,\hat{y}_j$ & Inherited coordinates of the bottom-left corner of piece $j \in J^{\text{fix}}$, fixed by the predecessor SOPP\\
$\hat{\xi}_j$ & Inherited rotation of piece $j \in J^{\text{fix}}$, fixed by the predecessor SOPP\\
$\theta^\text{p}$ & Set of ordered pairs $(j, j')$ indicating NR-p precedence relations, where piece $j$ is loaded later than $j'$\\
$\theta^\text{d}$ & Set of ordered pairs $(j, j')$ indicating NR-d precedence relations, where piece $j$ is unloaded earlier than $j'$\\
\midrule
\multicolumn{2}{@{}l}{Decision Variables}\\
\midrule
$x_j,\, y_j$ & Coordinates of the bottom-left corner of piece $j$\\
$\xi_j$ & Binary rotation variable; $\xi_j = 1$ if piece $j$ is rotated by $90^\circ$, and $\xi_j = 0$ otherwise\\
$\bar{w}_j,\, \bar{l}_j$ & Effective width and length of piece $j$, accounting for rotation\\
\bottomrule
\end{tabularx}
\end{table}

\clearpage
\onehalfspacing

\begin{center}
    {\LARGE\bfseries Electronic Companion\par}
\end{center}

\vspace{1em}

\setcounter{section}{0}
\setcounter{figure}{0}
\setcounter{table}{0}
\setcounter{equation}{0}

\renewcommand{\thesection}{EC.\arabic{section}}
\renewcommand{\thefigure}{EC.\arabic{figure}}
\renewcommand{\thetable}{EC.\arabic{table}}
\renewcommand{\theequation}{EC.\arabic{equation}}

\section{Implementation of Descriptor-Based Candidate Retrieval}
\label{sec:supp_descriptor_retrieval}

This section provides the implementation details of the descriptor-based retrieval rule summarized in Section 5. The purpose of this rule is to avoid comparing an unverified SOPP with every verified SOPP in the repository. Instead, verified SOPPs are indexed by aggregate descriptors that provide inexpensive necessary conditions for dominance.

Each verified feasible SOPP is stored together with its packing plan and a small set of aggregate descriptors that provide inexpensive necessary conditions for dominance. Let $\bar\Psi$ denote the repository of verified feasible SOPPs. For a verified SOPP $\bar \psi \in \bar\Psi$, define
\begin{align*}
A(\bar\psi) := \sum_{i\in\bar\psi} w_i h_i,\, \quad
S(\bar\psi) := \max_{i\in\bar\psi} \max\{w_i,h_i\},\, \quad
M(\bar\psi) := \max_{i\in\bar\psi} w_i h_i,
\end{align*}
where $A(\bar\psi)$ is the total item area, $S(\bar\psi)$ is the maximum item side length, and $M(\bar\psi)$ is the maximum single-item area.

For an unverified SOPP $\psi$, a verified SOPP $\bar\psi$ cannot dominate $\psi$ if any of the inequalities
\[
A(\bar\psi) < A(\psi),\qquad
S(\bar\psi) < S(\psi),\qquad
M(\bar\psi) < M(\psi)
\]
holds. These descriptor comparisons therefore serve as an inexpensive prescreening step before detailed dominance verification.

To support efficient retrieval, the verified-SOPP repository is organized according to discretized descriptor values. Specifically, for each descriptor $X \in \{A,S,M\}$, let
\[
o_X(x) := \left\lfloor \frac{x}{\Delta_X} \right\rfloor,
\]
where $\Delta_X > 0$ is a prescribed discretization width. The repository is then partitioned into index classes 
\[
\mathcal{O}_X(\eta) := \{\bar\psi \in \bar\Psi : o_X(X(\bar\psi)) = \eta\},
\]
where $\mathcal{O}_X(\eta)$ contains the verified SOPPs whose descriptor $X$ falls into class $\eta$.

Given an unverified SOPP $\psi$, define the threshold indices
\[
\eta_A := o_A(A(\psi)),\qquad
\eta_S := o_S(S(\psi)),\qquad
\eta_M := o_M(M(\psi)).
\]
The retrieved candidate set for SOPP $\psi$ is then
\begin{align*}
\mathcal{C}(\psi)
:=
\left(\bigcup_{\eta\ge \eta_A}\mathcal{O}_A(\eta)\right)
\cap
\left(\bigcup_{\eta\ge \eta_S}\mathcal{O}_S(\eta)\right)
\cap
\left(\bigcup_{\eta\ge \eta_M}\mathcal{O}_M(\eta)\right).
\end{align*}
Hence, only verified SOPPs whose discretized descriptor levels are compatible with those of $\psi$ are retained for further dominance checking.

For the descriptor-based repository structure, the discretization widths are determined from the instance characteristics. Specifically, the bucket width for the total-area descriptor $A(\psi)$ is set to the larger of 100 and 20\% of the vehicle trunk area. The bucket widths for the maximum-side-length descriptor $S(\psi)$ and the maximum-single-item-area descriptor $M(\psi)$ are determined from the range of item dimensions and item areas, respectively, by dividing the corresponding range into 20 buckets. These rules provide a simple instance-dependent discretization scheme that avoids overly fine indexing while preserving sufficient resolution for candidate retrieval. In the baseline setting, the resulting discretization widths are $\Delta_A=160$, $\Delta_S=1$, and $\Delta_M=16$.

The resulting set $\mathcal{C}(\psi)$ contains all verified SOPPs that pass the descriptor-based necessary conditions and are therefore retained for detailed dominance checking. Since dominance screening terminates once a dominating verified SOPP is found, the order and number of retrieved candidates examined can further affect the computational effort. This motivates the threshold and hot store introduced in Section 5.

\end{document}